\documentclass[tbtags,reqno]{amsart}

\usepackage{geometry}
\usepackage{tikz}

\usepackage{xcolor}

\definecolor{supred}{RGB}{180,0,0}
\definecolor{supblue}{RGB}{0,60,180}

\usepackage{mathrsfs}
\usepackage{latexsym}
\usepackage{amsthm,amsopn,tabmac,amsfonts,amssymb,epsfig,color,pstricks,pb-diagram,mathdots,stmaryrd}
\usepackage[active]{srcltx}
\numberwithin{equation}{section}

\makeatletter
\renewcommand{\subsubsection}{\@startsection
{subsubsection}
{3}
{0mm}
{\baselineskip}
{-0.5\baselineskip}
{\normalfont\normalsize\bfseries}}
\makeatother

\newtheorem{theorem}{Theorem}
\newtheorem{lemma}[theorem]{Lemma}
\newtheorem{proposition}[theorem]{Proposition}
\newtheorem{example}[theorem]{Example}
\newtheorem{conjecture}[theorem]{Conjecture}

\newtheorem{definition}[theorem]{Definition}

\newtheorem{remark}[theorem]{Remark}

\def\cal L{{\mathcal L}}

\def\setSD  {{\mathcal S}}

\newcommand{\cercle}[1]{\ensuremath{\setlength{\unitlength}{1ex}\begin{picture}(2.8,1.5)\put(1.4,0.7){\circle{2.8}\makebox(-5.6,0){#1}}\end{picture}}}
\newcommand{\tcercle}[1]{\ensuremath{\setlength{\unitlength}{1ex}\begin{picture}(2.8,2.8)\put(1.4,1.4){\circle{2.8}\makebox(-5.6,0){#1}}\end{picture}}}

\def\beq{\begin{equation}}
\def\eeq{\end{equation}}
\def\bea{\begin{align}}
\def\eea{\end{align}}

\title{A proof of the  $m$-Symmetric  Macdonald positivity
  at $t=1$}
\begin{document}

\author{Luc Lapointe}
\address{Instituto de Matem\'aticas, Universidad de
Talca, 2 norte 685, Talca, Chile.}
\email{llapointe@utalca.cl }

\author{Luis Pena}
\address{Instituto de Matem\'aticas, Universidad de
Talca, 2 norte 685, Talca, Chile.}
\email{luis.cardenas@utalca.cl }

\begin{abstract}
  We prove, in the case $t=1$, the extension to the $m$-symmetric world of the original Macdonald positivity conjecture.
  This is achieved by giving a combinatorial interpretation of the Kostka coefficients  $K_{\Omega \Lambda}(q,1)$  in terms of standard fillings of the diagram associated to the $m$-partition $\Omega$. This interpretation generalizes the one in the usual Macdonald case, which is given by a major index statistic on standard tableaux.
\end{abstract}
  
\keywords{Macdonald polynomials, Schur functions, symmetric functions}

\thanks{Funding: this work was supported by
the Fondo Nacional de Desarrollo Cient\'{\i}fico y Tecnol\'ogico de Chile (FONDECYT) Regular Grant \#1250376 and by the Beca de Doctorado Nacional ANID \#21211108.}

\maketitle

\section{Introduction}	
For a nonnegative integer $m$,  
the ring $R_m$ of $m$-symmetric functions is the subring of $\mathbb Q(q,t)[[x_1,x_2,x_3,\dots]]$ consisting of formal power series that are symmetric in the variables $x_{m+1},x_{m+2},x_{m+3},\dots$ (the first $m$ variables thus playing a distinguished, non-symmetric role).  An extension of the Macdonald polynomials $P_\Lambda(x;q,t)$ to this setting,
indexed by $m$-partitions and forming a basis of $R_m$, was introduced in  \cite{L, CL}. The same family has also been considered in \cite{BG,BG2} under the name of {\it partially-symmetric Macdonald polynomials}. Remarkably, it has been shown in \cite{OW} that the partially-symmetric/$m$-symmetric Macdonald polynomials are in correspondence with certain $(\mathbb C^*)^2$-fixed point classes of the parabolic flag Hilbert schemes \cite{GM}, a result that generalizes the correspondence between Macdonald polynomials and  $(\mathbb C^*)^2$-fixed points of the Hilbert schemes \cite{Haiman}.

In \cite{L, CL, BG,BG2,OW}, the partially-symmetric/$m$-symmetric Macdonald polynomials are defined as a $t$-symmetrization of the non-symmetric Macdonald polynomials\footnote{The version of the non-symmetric Macdonald polynomials used in  \cite{BG,BG2,OW}, although equivalent, differs slightly from that used in   \cite{L, CL}.  As a consequence, the $t$-symmetrization in \cite{BG,BG2,OW} is performed on all but the {last} $m$ variables, whereas in \cite{L,CL} it is performed on all but the {first} $m$ variables.}. In this article, we instead rely on the results of \cite{CL} to define the $m$-symmetric Macdonald polynomials through an
orthogonality/unitriangularity  characterization akin to that of the usual Macdonald polynomials. Specifically, there is a natural scalar product on $R_m$ with respect to which the  power-sum symmetric functions in the $m$-symmetric world are orthogonal: 
\begin{equation*} 
\langle p_\Lambda(x;t)\, ,\, p_\Omega(x;t)  \rangle_m =\delta_{\Lambda \Omega} \, q^{|\pmb a|}t^{{\rm Inv} (\pmb a)} z_\lambda(q,t).
\end{equation*}
The $m$-symmetric Macdonald polynomials are then the 
  unique basis $\{ P_\Lambda(x;q,t)\}_\Lambda$ of $R_m$ such that:
  \begin{enumerate}
\item $\displaystyle{\langle P_\Lambda(x;q,t)\, ,\, P_\Omega(x;q,t)  \rangle_m = 0 {\rm ~if~} \Lambda\neq\Omega}  $ \quad {\rm (orthogonality)}    
\item $\displaystyle{P_\Lambda(x;q,t)= m_\Lambda + \sum_{\Omega < \Lambda} d_{\Lambda \Omega}(q,t) \, m_\Omega}$ \quad {\rm (unitriangularity)},
  \end{enumerate}
where the $m_\Lambda$'s are the $m$-symmetric monomials and where the order on $m$-partitions is an appropriate generalization of the usual dominance ordering.

A central feature of the $m$-symmetric framework, introduced in \cite{L}, is an extension of the original Macdonald positivity conjecture (now a theorem \cite{GH,Haiman}).  The conjecture asserts that
the plethystically modified integral form  $J_\Lambda(x;q,t)=c_\Lambda(q,t) P_\Lambda(x;q,t)$
of the $m$-symmetric Macdonald polynomials expands positively in terms of the
$m$-symmetric Schur functions (which now depend on a parameter $t$):
\begin{equation} \label{Macdoposintro}
  J_\Lambda\left[\frac{X}{1-t};q,t \right] = \sum_\Omega K_{\Omega \Lambda}(q,t) \, s_\Omega[X;t]
 \qquad {\rm with~} K_{\Omega \Lambda}(q,t) \in \mathbb N[q,t].
\end{equation}
The motivation for introducing this conjecture was to
embed the $(q,t)$-Kostka coefficients into a much richer combinatorial framework, where the additional structure could shed light on their notoriously elusive nature.  In particular, this extra structure appears to include special recursions and Butler-type rules for the coefficients $K_{\Omega \Lambda}(q,t)$, which could
hold the key to a long-sought combinatorial interpretation of the $(q,t)$-Kostka coefficients in terms of a $q,t$-statistic on standard tableaux.

The main goal of this article is to prove the $m$-symmetric Macdonald positivity conjecture at $t=1$.
In the usual Macdonald case \cite{M}, one typically begins with the $q=1$ case (in which the Macdonald polynomial reduces to an elementary symmetric function) and extracts the $t=1$ case from the duality $K_{\mu \lambda}(q,t)=K_{\mu' \lambda'}(t,q)$.  Alternatively, the $t=1$ case can be obtained directly, using the factorization
$$
\lim_{t \to 1} J_\lambda\left[\frac{X}{1-t};q,t\right]
=(q;q)_\lambda h_\lambda\left[\frac{X}{1-q}\right], 
 $$
where $(q;q)_\lambda$ is a product of $q$-shifted factorials $(q;q)_\ell$, and where $h_\lambda$ is a homogeneous symmetric function.  Combined with the identity
$$
(1-q^\ell) h_\ell\left[ \frac{X}{1-q} \right] = \sum_{k=0}^{\ell-1} q^k h_{\ell-k}[X] h_{k}\left[\frac{X}{1-q}\right],
$$
this leads, by induction and the Pieri rules, to
\begin{equation} \label{usualcase}
K_{\mu \lambda}(q,1) = \sum_P q^{{\rm maj}_\lambda(P)} ,
\end{equation}
where the sum is over all standard tableaux $P$ of shape $\mu$, and where ${\rm maj}_\lambda$ is a major index statistic depending on $\lambda$.

Since the duality $K_{\mu \lambda}(q,t)=K_{\mu' \lambda'}(t,q)$ has no analog in the $m$-symmetric world, our strategy is to generalize the direct approach at $t=1$.  We will show, in Proposition~\ref{propJ1}, that the polynomial $J_\Lambda\left[{X}/{(1-t)};q,t \right]$ also factorizes at $t=1$:
\begin{equation}  \label{Jinhintro}
\lim_{t \to 1} J_\Lambda\left[\frac{X}{1-t};q,t\right]
=(q;q)_\Lambda
 h_{a_1}\left[x_1+\frac{qX}{1-q} \right]\cdots  h_{a_m}\left[x_m+\frac{qX}{1-q} \right]  h_\lambda\left[\frac{X}{1-q}\right] . 
\end{equation}
 The proof of this non-trivial factorization relies on properties of the scalar product $\langle \cdot, \cdot \rangle_{q,t}$ together with the explicit formula for the squared norm $\langle P_\Lambda(x;q,t), P_\Lambda(x;q,t) \rangle_{q,t}$ obtained in \cite{CL}.

 Combining \eqref{Jinhintro} with the identity
 $$
h_{\ell} \left[x_i+\frac{qX}{1-q} \right] = \sum_{k=0}^{\ell} q^k x_i^{\ell-k} h_{k} \left[\frac{X}{1-q} \right]  
 $$
yields our main result (Theorem~\ref{theo}), namely the sought-after combinatorial interpretation:
\begin{equation} \label{eqKosIntro}
K_{\Omega \Lambda}(q,1) = \sum_T q^{{\rm maj}_{\tilde \Lambda}(T_{\pmb a})},
\end{equation}
where the sum is over all standard fillings $T$ of the shape $\Omega$, where ${\rm maj}_{\tilde \Lambda}$ is a major index statistic  depending on $\Lambda$, and where $T_{\pmb a}$ is a standard tableau constructed from $T$ and the non-symmetric entries $\pmb a$ in $\Lambda=(\pmb a;\lambda)$.  Owing to the intricacies of the $m$-symmetric Schur functions, 
this formula applies directly only when $\Omega$ is dominant (that is, when the entries $\pmb b=(b_1,\dots,b_m)$ in $\Omega=(\pmb b;\mu)$ satisfy $b_1 \geq b_2 \geq \cdots \geq b_m$).  This restriction is easily lifted, however, using the symmetry (Proposition~\ref{propsym})
\begin{equation} \label{eqsymintro}
 K_{\Omega \Lambda} (q,1) =K_{\sigma(\Omega) \sigma(\Lambda)} (q,1),
\end{equation}
which holds for any permutation $\sigma \in S_m$, where $\sigma(\Lambda)=(a_{\sigma^{-1}(1)},\dots, a_{\sigma^{-1}(m)};\lambda)$ if $\Lambda=(\pmb a;\lambda)$.
A noteworthy corollary of \eqref{eqKosIntro} is that
$$
K_{\Omega \Lambda}(1,1) = \# \{ {\rm standard~fillings~of~the~shape~} \Omega \}=  \# \{ {\rm standard~tableaux~of~shape~} \pmb b \cup \mu \}. 
$$
This makes it clear that any combinatorial interpretation of the $K_{\Omega \Lambda}(q,t)$ coefficients --- which contain the usual $(q,t)$-Kostka coefficients $K_{\mu \lambda}(q,t)$
as special cases --- must amount to producing a $q,t$-statistic on standard tableaux.

The article is organized as follows.
Section~\ref{sect2} presents the extension to the $m$-symmetric world of the basic concepts in symmetric function theory.  Section~\ref{secmsymMacod} introduces the orthogonality/unitriangularity characterization of the $m$-symmetric Macdonald polynomials, which relies on the Hecke algebra in the definition of the $t$-deformation of the $m$-symmetric power-sums.  The (dual) $m$-symmetric Schur functions are defined in Section~\ref{secdual}, which sets the stage for the statement of the $m$-symmetric Macdonald positivity conjecture in Section~\ref{secMain}.  Section~\ref{sectab} establishes a tableau formula for the expansion of the dual $m$-symmetric Schur functions; by duality, this yields a combinatorial description of the product of a dominant monomial and a Schur function in terms of $m$-symmetric Schur functions.  The factorization \eqref{Jinhintro} and the symmetry \eqref{eqsymintro}
are proved in Section~\ref{sect1Mac}.  Finally, Section~\ref{sect8} develops the basic properties of the Jeu de Taquin that we will need and establishes the combinatorial interpretation \eqref{eqKosIntro} for the coefficients $K_{\Omega \Lambda}(q,1)$.

\section{The ring of $m$-symmetric functions} \label{sect2}

Most of this section is taken from \cite{L}.
Let $\mathbf \Lambda=\mathbb Q(q,t)[h_1,h_2,h_3,\dots]$ be the ring of symmetric functions in the variables $x_1,x_2,x_3,\dots$ (the standard references on symmetric functions are \cite{M,Stan}), where
$$
h_r=h_r(x_1,x_2,x_3,\dots) = \sum_{i_1 \leq i_2 \leq \cdots \leq i_r} x_{i_1} x_{i_2} \cdots x_{i_r} .
$$
 Bases of $\mathbf \Lambda$ are 
indexed by partitions $\lambda=(\lambda_1\geq\dots\geq\lambda_k>0)$ 
whose degree $\lambda$ is $|\lambda|=\lambda_1 +\cdots +\lambda_k$
and whose length $\ell(\lambda)=k$.  Each partition $\lambda$ has an 
associated Young diagram with $\lambda_i$ lattice squares in the $i^{th}$ 
row, from top to  bottom (English notation).  Any lattice square $(i,j)$ 
in the $i$th row and $j$th column of a Young diagram is called a cell.
 The partition
$\lambda\cup\mu$ is the non-decreasing rearrangement of the parts 
 of $\lambda$ and $\mu$.
  The dominance order $\geq$ is defined on partitions by
$\lambda\geq \mu$ when $\lambda_1+\cdots+\lambda_i\geq
\mu_1+\cdots+\mu_i$ for all $i$, and $|\lambda|=|\mu|$.

We define
the ring $R_m$ of $m$-symmetric functions as the subring of $\mathbb Q(q,t)[[x_1,x_2,x_3,\dots]]$ 
made  of formal power series that are symmetric  in the variables $x_{m+1},x_{m+2},x_{m+3},\dots$.
In other words, we have
$$R_m \simeq \mathbb Q(q,t)[x_1,\dots,x_m] \otimes \mathbf \Lambda_m,$$
where
$\mathbf \Lambda_m$ is the ring of symmetric functions in the variables $x_{m+1},x_{m+2},x_{m+3},\dots$.
It is  immediate that $R_0=\mathbf \Lambda$ is the usual ring of symmetric functions and that $R_0 \subseteq R_1 \subseteq R_2 \subseteq \cdots $.  Bases of $R_m$ are naturally indexed by $m$-partitions which are pairs $\Lambda=(\pmb a;\lambda)$, where 
$\pmb a= (a_1,\dots,a_m) \in \mathbb Z_{\geq 0}^m$ is a composition with $m$ parts, and where
$\lambda$ is a partition.  We will call the entries of $\pmb a$ and $\lambda$ the
non-symmetric and symmetric entries of $\Lambda$ respectively.
In the following, unless stated otherwise, $\Lambda$ and $\Omega$ will always stand respectively for the $m$-partitions $\Lambda=(\pmb a;\lambda)$ and $\Omega=(\pmb b;\mu)$. Observe that we use a different notation for  the composition $\pmb a$ with $m$ parts (which corresponds to the non-symmetric entries 
of $\Lambda$)
than for the composition $\eta$ with $N$ parts (which will typically index a non-symmetric Macdonald polynomial).

Given a composition $\pmb a$ and a partition $\lambda$, $\pmb a \cup \lambda$ will denote the partition obtained by reordering the entries of the concatenation of $\pmb a$ and $\lambda$.
The degree of an $m$-partition $\Lambda$, denoted $|\Lambda|$, is the sum of the degrees of $\pmb a$ and $\lambda$, that is, 
$|\Lambda|=a_1+\dots +a_m+\lambda_1+\lambda_2+\cdots$. We also define the 
length of $\Lambda$ as $\ell(\Lambda)=m+\ell(\lambda)$. We will say that $\pmb a$ is dominant if $a_1 \geq a_2 \geq \cdots \geq a_m$, and by extension, we will say that $\Lambda=(\pmb a; \lambda)$ is dominant if $\pmb a$ is dominant.  If $\pmb a$ is not dominant, we let $\pmb a^+$ be the dominant composition obtained by reordering the entries of $\pmb a$.

There is a natural way to represent an $m$-partition by a Young diagram. 
The diagram corresponding to $\Lambda$ is the Young diagram of $\pmb a \cup \lambda$ with an $i$-circle  added to the right of the row of size $a_i$ for $i=1,\dots,m$ (if there are many rows of size $a_i$, the circles are ordered from top to bottom in increasing order).  For instance, given $\Lambda=(2,0,2,1; 3,2 )$, we have
 $$
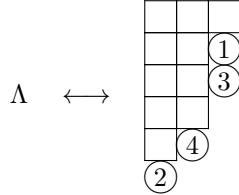
\Lambda \quad \longleftrightarrow  \quad {\tableau[scY]{&& & \bl \\& & \bl \cercle{1} \\& & \bl \cercle{3}\\ & & \bl  \\ & \bl \cercle{4} \\ \bl \cercle{2} }}
$$
Observe that when $m=0$, the diagram associated to  $\Lambda=( ; \lambda)$ coincides with the Young diagram associated to $\lambda$.
Also note that if $\eta$ is a composition with $m$ parts, then the diagram of $\eta$ coincides with the diagram of the $m$-partition $\Lambda=(\pmb a;\emptyset)$, where $\pmb a=\eta$.  We let $\Lambda^{(0)}=\pmb a \cup \lambda$, that is, $\Lambda^{(0)}$ is the partition 
obtained from the diagram of $\Lambda$ by discarding all the circles.
More generally, for $i=1,\dots,m$, we let $\Lambda^{(i)}=(\pmb a+1^i) \cup \lambda$, where $\pmb a +1^i=(a_1+1,\dots,a_i+1,a_{i+1},\dots,a_m)$. In other words, 
$\Lambda^{(i)}$ is the partition obtained from the diagram associated to $\Lambda$ by changing all of the $j$-circles, for $1 \leq j \leq i$, into squares and discarding the remaining circles.  Taking as above
$\Lambda=(2,0,2,1; 3,2)$, we have
$\Lambda^{(0)}=(3,2,2,2,1)$, $\Lambda^{(1)}=(3,3,2,2,1)$, $\Lambda^{(2)}=(3,3,2,2,1,1)$, $\Lambda^{(3)}=(3,3,3,2,1,1)$ and $\Lambda^{(4)}=(3,3,3,2,2,1)$. 
We then define the
dominance ordering on $m$-partitions to be  such that
\begin{equation} \label{deforder}
\Lambda \geq \Omega \iff \Lambda^{(i)} \geq \Omega^{(i)} \qquad {\rm for~all~}i=0,\dots,m,
\end{equation}
where the order on the right-hand-side
is the usual dominance order on partitions.

\begin{example}
We have that $\Lambda=(0,3;1) \geq (2,1;1) =\Omega$,
since $\Lambda^{(2)}=(4,1,1) \geq (3,2,1) =\Omega^{(2)}$, $\Lambda^{(1)}=(3,1,1) \geq (3,1,1) =\Omega^{(1)}$, and $\Lambda^{(0)}=(3,1) \geq (2,1,1) =\Omega^{(0)}$.
\end{example}

We will associate arm and leg-lengths to the cells of the diagram of an $m$-partition.  Because of the circles, we will need two notions of arm-lengths as well as two notions of leg-lengths.  The arm-length  $a(s)$ 
is equal to the number of cells in $\Lambda$ strictly to the right of $s$ (and in the same row).  Note that if there is a circle at the end of its row, then it adds one to the arm-length of $s$.
 The arm-length $\tilde a(s)$ is exactly as $a(s)$ except that the circle at the end of the row does not contribute to $\tilde a(s)$.

The leg-length $\ell(s)$ is equal to the number of cells in $\Lambda$ strictly below $s$ (and in the same column). If at the bottom of its column there are $k$ circles whose fillings are smaller than the filling of the circle at the end of its row, then they add $k$ to the value of the leg-length of $s$.  If the row does not end with a circle then none of the circles at the bottom of its column  contributes to the leg-length.
 The leg-length $\tilde \ell(s)$ is exactly as $\ell(s)$ except that the circles at the bottom of the column contribute to $\tilde \ell(s)$ when there is no circle at the end of the row of $s$.

\begin{example} \label{exal}
The values of $a(s)$ and $\ell(s)$ in each cell of the diagram of  $\Lambda={(2,0,0,2;4,1,1)}$ are
$$
 {\tableau[scY]{\mbox{\scriptsize{\rm 34}}&\mbox{\scriptsize{\rm 22}}&  \mbox{\scriptsize{\rm 10}} & \mbox{\scriptsize{\rm 00}} & \bl \\\mbox{\scriptsize{\rm 23}}& \mbox{\scriptsize{\rm 11}}& \bl \cercle{\rm 1} \\\mbox{\scriptsize{\rm 24}}& \mbox{\scriptsize{\rm 10}} & \bl \cercle{\rm 4}\\ \mbox{\scriptsize{\rm 01}}& \bl  \\ \mbox{\scriptsize{\rm 00}} \\ \bl \cercle{\rm 2} \\ \bl \cercle{\rm 3}}}
 $$
while those of 
 $\tilde a(s)$ and $\tilde \ell(s)$ are
 $$
 {\tableau[scY]{\mbox{\scriptsize{\rm 36}}&\mbox{\scriptsize{\rm 22}}&  \mbox{\scriptsize{\rm 12}} & \mbox{\scriptsize{\rm 00}} & \bl \\\mbox{\scriptsize{\rm 13}}& \mbox{\scriptsize{\rm 01}}& \bl \cercle{\rm 1} \\\mbox{\scriptsize{\rm 14}}& \mbox{\scriptsize{\rm 00}} & \bl \cercle{\rm 4}\\ \mbox{\scriptsize{\rm 03}}& \bl  \\ \mbox{\scriptsize{\rm 02}} \\ \bl \cercle{\rm 2} \\ \bl \cercle{\rm 3}}}
 $$
\end{example}

Let the $m$-symmetric monomial function $m_\Lambda(x)$ be defined as
$$
m_\Lambda(x) := x_1^{a_1} \cdots x_m^{a_m} \, m_\lambda (x_{m+1},x_{m+2},\dots)= x^{\pmb a}  \, m_\lambda (x_{m+1},x_{m+2},\dots),
$$
where $m_\lambda (x_{m+1},x_{m+2},\dots)$ is the usual monomial symmetric function
 in the variables $x_{m+1},x_{m+2},\dots$
$$
m_\lambda (x_{m+1},x_{m+2},\dots)= \sum_\alpha x_{m+1}^{\alpha_1} x_{m+2}^{\alpha_2} \cdots,   
$$
with the sum being over all distinct rearrangements $\alpha$ of $(\lambda_1,\lambda_2,\dots,\lambda_{\ell(\lambda)},0,0,\dots)$.
It is immediate that $\{ m_\Lambda(x) \}_\Lambda$ is a basis of $R_m$.

The $m$-symmetric power-sums are defined as
\begin{equation} \label{eqp}
p_\Lambda(x) := x_1^{a_1} \dots x_m^{a_m} \, p_\lambda (x)= x^{\pmb a} \, p_\lambda (x).
\end{equation}
 It should be observed that
 the variables in $p_\lambda$, contrary to those of $m_\lambda$ in $m_\Lambda(x)$, start at $x_1$ instead of
$x_{m+1}$. In this expression,  $p_\lambda (x)$ is the usual power-sum symmetric function
$$
p_\lambda(x) = \prod_{i=1}^{\ell(\lambda)} \, p_{\lambda_i}(x),
$$
where $p_r(x)=x_1^r+x_2^r+\cdots$.

 Let  $s_\lambda(x)$ be the Schur function indexed by the partition $\lambda$, which can be defined through the Jacobi-Trudi determinant:
\begin{equation} \label{eqSchur}
s_{\lambda}(x)=\det \Bigl( h_{\lambda_i-i+j}(x) \Bigr)_{1 \leq i,j \leq \ell(\lambda)},
\end{equation}
where $h_k(x)=0$ if $k<0$. By replacing $p_\lambda(x)$ by $s_\lambda(x)$ in \eqref{eqp}, we get another basis of $R_m$:
\begin{equation} \label{eqk}
k_\Lambda(x) := x^{\pmb a} s_\lambda(x).
\end{equation}
The basis $k_\Lambda(x)$ will play a fundamental role in this article.

\section{$m$-symmetric Macdonald polynomials} \label{secmsymMacod}
In order to define the $m$-symmetric Macdonald polynomials, we first need to define a $t$-modification of the $m$-symmetric  power-sum basis.  It will rely on the Hecke algebra.

Let the exchange operator $K_{i,j}$ be such that
$$K_{i,j} f(\dots, x_i,\dots,x_j,\dots)= f(\dots, x_j,\dots,x_i,\dots).$$
We then define the generators $T_i$, for $i=1,\dots,m-1$,
of the Hecke algebra as
\begin{equation} \label{eqTi}
  T_i=t+\frac{tx_i-x_{i+1}}{x_i-x_{i+1}}(K_{i,i+1}-1).
  \end{equation}
The $T_i$'s satisfy the relations:
\begin{align*} &(T_i-t)(T_i+1)=0\nonumber\\
&T_iT_{i+1}T_i=T_{i+1}T_iT_{i+1}\nonumber\\
&T_iT_j=T_jT_i \, ,\quad {\rm if~}|i-j| > 1.
\end{align*}

Let $H_{\pmb a}(x;t) = H_{\pmb a}(x_1,\dots,x_m;t)$ be the  non-symmetric Hall-Littlewood polynomial. The polynomial  $H_{\pmb a}(x;t)$ can be constructed recursively as follows.  If 
$\pmb a$ is dominant then $H_{\pmb a}(x;t)=x^{\pmb a}$.  Otherwise,
$T_i H_{\pmb a}(x;t)=H_{s_i \pmb a}(x;t)$ if $a_i > a_{i+1}$ (with $s_i \pmb a=(a_1\dots,a_{i+1},a_i,\dots,a_m)$).  Since $H_{\pmb a}(x;1)=x^{\pmb a}$,
the following $t$-deformation of the $m$-symmetric power sum basis
\begin{equation} \label{eqplambda}
p_\Lambda(x;t)= H_{\pmb a}(x;t) p_\lambda(x)
\end{equation}
also provides a basis of $R_m$.

Let $|\pmb a|=a_1+\cdots +a_m$, and let
${\rm Inv}(\pmb a)$ be the number of inversions in $\pmb a$, that is,
$$
{\rm Inv}(\pmb a)=\#\{(i,j) \, | \, 1\leq i<j\leq m {\rm ~and~} a_{i} < a_j  \}.
$$
We now introduce a  scalar product in $R_m$ defined on the $t$-deformation of the $m$-symmetric power-sums as:
\begin{equation} \label{eqscal}
\langle p_\Lambda(x;t)\, ,\, p_\Omega(x;t)  \rangle_{q,t} =\delta_{\Lambda \Omega} \, q^{|\pmb a|}t^{{\rm Inv} (\pmb a)} z_\lambda(q,t),
\end{equation}
where
$$
z_\lambda(q,t) = z_\lambda \prod_{i=1}^{\ell(\lambda)} \frac{1-q^{\lambda_i}}{1-t^{\lambda_i}}.
$$
In the previous expression,   $z_\lambda= \prod_{i \geq 1 } i^{n_\lambda(i)} \cdot n_\lambda(i)!$, with $n_\lambda(i)$ the number of occurrences of $i$ in $\lambda$. Observe that
when $m=0$, this corresponds to the usual Macdonald polynomial scalar product \cite{M}.

The $m$-symmetric Macdonald polynomials can be defined with the following
orthogonality/unitriangularity  characterization akin to that of the usual Macdonald polynomials.
\begin{proposition} \label{propdefmacdo} The $m$-symmetric Macdonald polynomials form the 
  unique basis $\{ P_\Lambda(x;q,t)\}_\Lambda$ of $R_m$ such that
  \begin{enumerate}
\item $\displaystyle{\langle P_\Lambda(x;q,t)\, ,\, P_\Omega(x;q,t)  \rangle_{q,t} = 0 {\rm ~if~} \Lambda\neq\Omega}  $ \quad {\rm (orthogonality)}    
\item $\displaystyle{P_\Lambda(x;q,t)= m_\Lambda + \sum_{\Omega < \Lambda} d_{\Lambda \Omega}(q,t) \, m_\Omega}$ \quad {\rm (unitriangularity)}
  \end{enumerate}
for certain coefficients $d_{\Lambda \Omega}(q,t) \in \mathbb Q(q,t)$. We recall that the dominance order on $m$-partitions was defined in \eqref{deforder}.  
\end{proposition} 
We will later need the squared norm of an $m$-symmetric Macdonald polynomial, which is given explicitly as \cite{CL}
\begin{equation} \label{squared}
\langle P_\Lambda(x;q,t)\, ,\, P_\Lambda(x;q,t)  \rangle_{q,t} =  q^{|\pmb a|}t^{{\rm Inv} (\pmb a)} \prod_{s \in \Lambda} \frac{1 -q^{\tilde a(s)+1}t^{\tilde \ell(s)}}{1 -q^{a(s)}t^{\ell(s)+1}}\, .
\end{equation}

\section{The (dual) $m$-symmetric Schur functions} \label{secdual}
We are now ready to introduce the (dual) $m$-symmetric Schur functions.
We will first introduce the dual $m$-symmetric Schur functions.  Then, by duality, the  
$m$-symmetric Schur functions will be defined.
In the following, it will prove convenient to use the plethystic notation \cite{Hag} in which, for a symmetric function $f$ and $X=x_1+x_2+\cdots$, we let $f[X]=f(x)=f(x_1,x_2,\dots)$.  More generally, we have using this notation that $f[X+x_1+\cdots+x_k]=f(x_1,\dots,x_k,x_1,x_2,\dots)$.

Let $\nu$ be a partition of length $\ell$.  For a sequence of alphabets $X_1,\dots,X_\ell$, where $\ell=\ell(\nu)$, the multi-Schur function $s_\nu(X_1,\dots,X_\ell)$ is defined as \cite{MacS}
\begin{equation} \label{eqmulti}
s_\nu(X_1,\dots,X_\ell) = \det \Bigl( h_{\nu_i-i+j}[X_i] \Bigr)_{1 \leq i,j \leq \ell}.
\end{equation}
Observe, from \eqref{eqSchur},
that $s_\nu(X_1,\dots,X_\ell)$ is equal to the usual Schur function
$s_\nu(x)$ whenever $X=X_1=X_2=\cdots=X_\ell$.

\begin{definition} \label{defdualS}
  The dual $m$-symmetric Schur functions $ s_{\Lambda}^*(x;t)$
  are defined recursively in the following way.  If $\Lambda=(\pmb a;\lambda)$ is dominant then
$$
  s_\Lambda^*(x;t)=s_\nu(X_1,\dots,X_\ell),
$$
where $\nu=\Lambda^{(0)}=\pmb a \cup \lambda$, and where
 $X_i$ stands for the alphabet $X+x_1+\cdots +x_k$ with $k$ the number of circles weakly above row $i$ in the diagram corresponding to $\Lambda$.
 Otherwise, if $a_i<a_{i+1}$  then
 \begin{equation} \label{recursis}
s^*_\Lambda(x;t)= T_i s_{\tilde \Lambda}^*(x;t),
 \end{equation}
where $\tilde \Lambda=s_i \Lambda$. This amounts to saying that
\begin{equation} \label{stsigma}
s^*_\Lambda(x;t)=  T_{\sigma^{-1}} s_{\Lambda^+}^*(x;t),
\end{equation}
where $\sigma$ is the shortest permutation such that $\sigma(\pmb a)=(a_{\sigma^{-1}(1)},\dots,a_{\sigma^{-1}(m)}) = \pmb a^+$. 
\end{definition}

\begin{example} The diagram associated to $(2,1;3,1)$ is 
$$
{\tableau[scY]{&& & \bl \\& & \bl \cercle{\rm 1} \\&  \bl \cercle{\rm 2}\\ & \bl  \\ }}
$$
from which we deduce that
   $$
s_{2,1;3,1}^*(x;t)= \left|
\begin{array}{cccc}
  h_3[X] & h_4[X] & h_5[X] & h_6[X ]\\
   h_1[X+x_1 ]& h_2[X+x_1] & h_3[X+x_1] & h_4[X+x_1] \\
   0 & h_0[X+x_1+x_2]  & h_1 [X+x_1+x_2]& h_2[X+x_1+x_2]\\
   0 & 0 & h_0[X+x_1+x_2] & h_1[X+x_1+x_2]
\end{array}   
\right| .
$$
\end{example}
\begin{remark} \label{remarkflag}
The multi-Schur functions that we use in Definition~\ref{defdualS} 
are essentially flagged Schur functions \cite{LS,W} in infinite alphabets instead of finite alphabets. For instance, if $X$ were equal to $y_{1}+\cdots+y_j$, the dual $m$-symmetric Schur function $s_{2,1;3,1}^*(x;t)$ would correspond in the language of \cite{W} to the flagged Schur functions $s_{3,2,1,1}(b)$ with flags $b_1=0,b_2=1,b_3=2$ and $b_4=2$ (with the understanding that 
the variables $y_1,\dots,y_j$
would not be constrained by any of the flags).
\end{remark}

A bilinear scalar product $\langle \cdot, \cdot \rangle_m$ on $R_m$ is defined by requiring that the power-sum basis be such that
\begin{equation} \label{scalprod}
\langle p_\Lambda(x;t),  p_\Omega(x;t)  \rangle_m = \delta_{\Lambda \Omega} t^{{\rm Inv} (\pmb a)} z_\lambda ,
\end{equation}
where we recall that
${\rm Inv} (\pmb a)$
is the number of inversions in $\pmb a$.
It can be shown that the dual $m$-symmetric Schur functions form a basis of $R_m$ \cite{L}.  The $m$-symmetric Schur functions $s_\Lambda(x;t)$ can thus be defined as the unique basis of $R_m$ such that
\begin{equation} \label{scalprodS}
\langle s_\Lambda(x;t),  s^*_\Omega(x;t)  \rangle_m = \delta_{\Lambda \Omega} t^{{\rm Inv} (\pmb a)}.
\end{equation}
For  $i=1,\dots,m-1$, the operator $T_i$ has a simple action on $s^*_\Lambda(x;t)$ by definition.  It also has the following simple action on $s_\Lambda(x;t)$:
\begin{equation} \label{Tis}
T_i  s_\Lambda(x;t) =
\left \{ \begin{array}{ll}
     s_{\tilde \Lambda}(x;t)  & {\rm if~} a_i > a_{i+1} \\
    (t-1)  s_{\Lambda}(x;t) +  t s_{\tilde \Lambda}(x;t)  & {\rm if~} a_i < a_{i+1} \\
     t s_{\Lambda}(x;t)  & {\rm if~} a_i = a_{i+1}
\end{array} \right .    ,
\end{equation}
where $\tilde \Lambda=(a_1,\dots,a_{i+1},a_i,\dots,a_m;\lambda)$.

\section{The $m$-symmetric Macdonald positivity conjecture} \label{secMain}
A positivity conjecture for the $m$-symmetric Macdonald polynomials in terms of $m$-symmetric Schur functions akin to the original Macdonald positivity conjecture \cite{M} (now theorem \cite{Haiman}) was stated in \cite{L}.
It relies on an extension of the notion of plethysm to $R_m$.
Recall that the plethysm relevant to Macdonald polynomials is the linear map on the ring of symmetric functions that sends the power-sum 
$p_\lambda$ to $p_\lambda /\prod_i (1-t^{\lambda_i})$ \cite{Ber}.
In the plethystic notation, this map is denoted as
$$
p_\lambda\left[\frac{X}{1-t}\right]=\frac{p_\lambda[X]}{\prod_i (1-t^{\lambda_i})}=\frac{p_\lambda(x)}{\prod_i (1-t^{\lambda_i})}.
$$
The notion of plethysm that we will need will simply be the linear map $R_m \to R_m$  defined on the $m$-symmetric power-sums $p_\Lambda(x;t)$ as
\begin{equation} \label{plethysm}
p_\Lambda\left[\frac{X}{1-t} ;t \right]= \frac{p_\Lambda[X;t] }{\prod_i (1-t^{\lambda_i})} =  \frac{p_\Lambda(x;t) }{\prod_i (1-t^{\lambda_i})}.
\end{equation}
We stress that the plethysm only depends on the symmetric part $\lambda$ of $\Lambda=(\pmb a; \lambda)$.

The integral form of the $m$-symmetric Macdonald polynomials is given by
\begin{equation} \label{intform}
J_\Lambda(x;q,t)= c_\Lambda(q,t) P_\Lambda(x;q,t),
\end{equation}
where
$$
c_\Lambda(q,t)= \prod_{s \in \Lambda} (1 -q^{a(s)}t^{\ell(s)+1}).
$$
Recall that the arm and leg-lengths were introduced in Section~\ref{sect2}.

It appears that the  $m$-symmetric Macdonald polynomials in their integral form
are, once plethystically transformed, positive in terms of the $m$-symmetric Schur functions.
\begin{conjecture} \label{mainconjec}
  The $m$-symmetric Macdonald polynomials in their integral form are such that
  \begin{equation} \label{Kostka}
  J_\Lambda\left[\frac{X}{1-t};q,t \right]=
  \sum_{\Omega} K_{\Omega \Lambda}(q,t) \, s_\Omega(x;t),
  \end{equation}
with $K_{\Omega \Lambda}(q,t) \in \mathbb N[q,t]$.
\end{conjecture}
 \begin{example} \label{expos}We have
\begin{align*}
  J_{1,0;2}\left[\frac{X}{1-t};q,t \right] & = t^2 s_{3,0;\emptyset} + qt^2 s_{0,3;\emptyset}+ q t \, s_{0,0;3} +(qt^3+t) s_{2,1;\emptyset}+(qt^2+t) s_{2,0;1}
  +(q^2t^3+t) s_{1,2;\emptyset} \\  & \qquad +(q^2t^2+1) s_{1,0;2}+(q^2t^2+qt) s_{0,2;1}+(q^2t^2+q) s_{0,1;2}+(q^2t+q) s_{0,0;2,1} \\
  & \qquad \qquad + qt^2 s_{1,1;1}+ qt \, s_{1,0;1,1} + q^2 t \, s_{0,1;1,1}+q^2s_{0,0;1,1,1}.
\end{align*}
\end{example}

In Section~\ref{sect8}  we will give a proof of Conjecture~\ref{mainconjec} in the case $t=1$ by providing
a combinatorial interpretation for the coefficients $K_{\Omega \Lambda}(q,1)$.
The combinatorial interpretation will imply 
in particular that, as can be appreciated in Example~\ref{expos},
$$
K_{\Omega \Lambda}(1,1) = \# \text{ of standard tableaux of shape } \mu \cup \pmb b.
$$

\section{A tableau expansion and the (dual) $m$-symmetric Schur functions at $t=1$.} \label{sectab}
In this section, we will provide a combinatorial interpretation for the expansion of the dual $m$-symmetric Schur functions (in the dominant case)  in  the $k_\Lambda(x)$ basis defined in \eqref{eqk}.
This combinatorial interpretation is essentially a rewriting of a result on flagged Schur functions \cite{W}.

A (skew) tableau $T$ of shape $\lambda/\mu$ is a filling of the skew diagram  $\lambda/\mu$  with integers such that the entries are strictly increasing along columns (from top to bottom) and weakly increasing along rows (from left to right). We first define an important set of tableaux.
\begin{definition} \label{defsets}
  For the $m$-partitions $\Lambda=(\pmb a;\lambda)$ and $\Omega=(\pmb b;\mu)$  of the same degree,  we define  $\setSD_{\Lambda \Omega}$ as the set of skew tableaux $T$ of shape $(\pmb a \cup \lambda)/\mu$ such that
the letter $i$ appears $b_i$ times and always within the first $a_i$ columns of $T$.     
\end{definition}
\begin{example} \label{exampleS}
  If $\Lambda=(4,4,2;3,2,1)$ and $\Omega=(1,3,1;4,3,2,1,1)$, the set
   $ \setSD_{\Lambda \Omega}$ contains the skew tableaux
  \begin{equation*}
 T_1=   {\small{\tableau[scY]{&& &\\& & & 2\\& & 2 \\ & 2  \\ & 3\\ 1 \\ }}} \qquad
  T_2=  {\small{\tableau[scY]{&& &\\& & & 2\\& & 2 \\ & 1  \\ & 3\\ 2 \\ }}} \qquad
  T_3=  {\small{\tableau[scY]{&& &\\& & & 2\\& & 1 \\ & 2  \\ & 3\\ 2 \\ }}} \qquad
  T_4=  {\small{\tableau[scY]{&& &\\& & & 1\\& & 2 \\ & 2  \\ & 3\\ 2 \\ }}} \qquad
   T_5=   {\small{\tableau[scY]{&& &\\& & & 2\\& & 2 \\ & 1  \\ & 2\\ 3 \\ }}} 
    \end{equation*}
\end{example}  
\begin{lemma} \label{lemmaskew}
  Suppose that  $T\in \setSD_{\Lambda \Omega}$ with
  $\Lambda$  dominant. If we let $T_\circ$ be obtained from $T$ by adjoining, for $i=1,\dots,m$, a square filled with the letter $i$ in the position of the 
  $i$-circle in $\Lambda$,
    then $T_\circ$ is a skew tableau.
\end{lemma}  
\begin{proof} Since $\Lambda^{(0)}=(\pmb a \cup \lambda)$
  and since the circles in $\Lambda$ are ordered from top to bottom in a given column, we only need to show that in $T$ there cannot be a letter larger than $i$ directly above the $i$-circle or directly to its left.  Observe that the $i$-circle lies in column $a_i+1$.  Since $\Lambda$ is dominant, a letter $j$ larger than $i$ can only occur in the first $a_j \leq a_i$ columns.  Hence the letter directly above the $i$-circle cannot be larger than $i$.  Now suppose that there is a letter $j>i$ directly to the left of the $i$-circle.  Let $c$ be the column in which the $i$-circle lies, and suppose that there are $k$ circles below the $i$-circle in column $c$. Since $\Lambda$ is dominant, this implies that there are only $k$ letters larger than $i$ that can appear in column $c-1$ of $T$.  But in column $c-1$ of $T$ there are at least $k$ cells below the cell in which the letter $j$ lies (because there are $k$ circles below the $i$-circle in column $c$ and those circles need to have a square to their left).  This is a contradiction since those $k$ cells cannot all be filled with letters larger than $j$ since $j>i$ and there are only $k$ letters larger than $i$.

\end{proof}

\begin{proposition} \label{dualSD} Let $\Lambda$ be dominant. Then
  \begin{equation} \label{eqD}
  s_{\Lambda}^*(x;t)= \sum_{\Omega} D_{\Lambda \Omega}  \, k_{\Omega}(x),
  \end{equation}
  where $D_{\Lambda \Omega}=\#\mathcal S_{\Lambda \Omega}$, and where we recall that
  $k_{\Omega}(x)$ was defined in \eqref{eqk}.
  \end{proposition}
\begin{proof}  From Definition~\ref{defdualS}, if  $\Lambda=(\pmb a;\lambda)$
is dominant then
$$
  s_\Lambda^*(x;t)=s_\nu(X_1,\dots,X_\ell),
$$
where $\nu=\Lambda^{(0)}=\pmb a \cup \lambda$, and where
 $X_i$ stands for the alphabet $X+x_1+\cdots +x_k$ with $k$ the number of circles weakly above row $i$ in the diagram corresponding to $\Lambda$.
Now, following Remark~\ref{remarkflag}, let $X=y_1+y_2+\cdots$ with the understanding that the variables $y_1,y_2,\dots$ are not constrained by any of the flags.  Suppose also that the letters $\hat 1, \hat 2,\dots$ and $1,\dots,m$ are ordered in the following way:
 $$
 \hat 1 < \hat 2 < \cdots < 1 <2 \cdots <m.
 $$
 Given a tableau $R$, let $(x,y)^R$ stand for the monomial with the letter $x_i$ (resp. $y_i$)
 having power $j$
 if $i$ (resp. $\hat i$) occurs $j$ times in $R$.
It is shown in  \cite{W} that
  $$
 s_{\nu}(X_1,\dots,X_\ell) = \sum_R (x,y)^R,
 $$
 where the sum is over all tableaux $R$ of shape $\nu$ in the letters $\hat 1, \hat 2,\cdots$ and $1,\dots,m$
 such that in the $i^{th}$ row the letters are all smaller than the number of circles weakly above row $i$.  Since $\Lambda$ is dominant, the circles  are ordered from 1 to $m$ when reading from top to bottom. This implies that the letter $k$ can lie in any row weakly below the row in which lies the circle $k$. Since the letters  larger than $k$ cannot be put in any row weakly above the row in which lies the circle $k$, we have, equivalently, that the letter $k$ can only lie within the first $a_k$ columns of $R$.

 Hence, given a tableau $R$ as above, the letters $\hat 1,\hat 2,\dots$ in $R$ form a tableau $Q$ of shape $\mu$, while the letters $1,2,\dots,m$ form a skew-tableau $T$ of shape $\nu/\mu$ where the letter $i$ appears always within the first $a_i$ columns of $T$. We thus have that
 $$
 s_{\nu}(X_1,\dots,X_\ell) = \sum_R (x,y)^R = \sum_{Q,T} y^Q x^T = \sum_{\mu} \sum_{T} s_{\mu}(y) x^T,
 $$
 where the last sum is over all skew-tableaux of shape  $\nu/\mu$ where the letter $i$ appears always within the first $a_i$ columns of $T$. Letting $y=x$ in the previous expression, this gives immediately that
 $$
 s_\Lambda^*(x;t)=
 s_{\nu}(X_1,\dots,X_\ell) = \sum_\Omega D_{\Lambda \Omega} k_\Omega(x),
 $$
where, for $\Omega=(\pmb b;\mu)$, $D_{\Lambda \Omega}$ is the number of tableaux $T$ of shape $\nu/\mu$ such that the letter  $i$ appears $b_i$ times and always within the first $a_i$ columns of $T$.  That is, $D_{\Lambda \Omega}=\#\mathcal S_{\Lambda \Omega}$. 
\end{proof}

At $t=1$, the scalar product \eqref{scalprod} is such that
\begin{equation} \label{eqdual}
\langle s_\Lambda(x;1), s^*_\Omega(x;1) \rangle_m^{t=1} = \delta_{\Lambda \Omega}
\qquad {\rm and}\qquad \langle k_\Lambda(x), k_\Omega(x) \rangle_m^{t=1} = \delta_{\Lambda \Omega} .
\end{equation}
If, for an arbitrary $\Lambda$, we let $D_{\Lambda \Omega}$ be such that
$$
s^*_\Lambda(x;1)= \sum_{\Omega} D_{\Lambda \Omega} k_{\Omega}(x),
$$
then the dualities in \eqref{eqdual} imply that
\begin{equation} \label{eqkins}
k_{\Omega}(x) = \sum_{\Lambda} D_{\Lambda \Omega} s_{\Lambda}(x;1).
\end{equation}
From Proposition~\ref{dualSD}, we have a combinatorial interpretation for the coefficient of $s_\Lambda(x;1)$ in the expansion of $k_{\Omega}(x)$ whenever $\Lambda$ is dominant. This will prove important in Section~\ref{sect8}.

\section{The case $t=1$ of the $m$-symmetric Macdonald polynomials} \label{sect1Mac}
We first prove that the functions
$$
h_\Lambda(x;q) :=  h_{a_1}\left[q^{-1}x_1+X \right]\cdots  h_{a_m}\left[q^{-1}x_m+X \right]  h_\lambda[X] 
$$
are dual to the monomial $m$-symmetric basis $\{m_\Lambda(x)\}$ with respect to the scalar product $\langle \cdot, \cdot \rangle'$ defined as
\begin{equation} \label{eqscalq}
\langle \, p_\Lambda(x)\, ,\, p_\Omega(x)  \, \rangle' =\delta_{\Lambda \Omega} \, q^{|\pmb a|}z_\lambda,
\end{equation}
where the $m$-symmetric power-sum basis  was introduced in \eqref{eqp}.
An elementary result in symmetric function theory states that \cite{M}
\begin{equation} \label{eqelem}
\frac{1}{ \prod_{j,k=1} (1-x_j y_k)  } = \sum_{\lambda} z_\lambda^{-1} p_\lambda(x) p_\lambda(y) = \sum_{\lambda} m_\lambda(x) h_\lambda(y). 
\end{equation}
Using
\begin{equation} \label{eqeasy}
 \frac{1}{\prod_{i=1}^m (1-q^{-1}x_i y_i)} = \sum_{\pmb a} q^{-|\pmb a|} x^{\pmb a} y^{\pmb a},
 \end{equation}
it then immediately follows that
\begin{equation} \label{kernel}
\frac{1}{\prod_{i=1}^m (1-q^{-1}x_i y_i)} \cdot \frac{1}{ \prod_{j,k=1} (1-x_j y_k)  } = \sum_\Lambda q^{-|\pmb a|}  z_\lambda^{-1} p_\Lambda(x) p_\Lambda(y),
\end{equation}
which means that the left-hand-side of the previous identity
is a reproducing kernel for the scalar product $\langle \cdot, \cdot \rangle'$.

\begin{proposition} \label{propmhq}
  We have that
$$
  \frac{1}{\prod_{i=1}^m (1-q^{-1}x_i y_i)} \cdot \frac{1}{ \prod_{j,k=1} (1-x_j y_k)  } = \sum_\Lambda m_\Lambda(x) h_\Lambda(y;q)
  $$
  or, equivalently, that
  $$
\langle \, m_\Lambda(x)\, ,\, h_\Omega(x;q)  \, \rangle' = \delta_{\Lambda \Omega}.
$$
\end{proposition}
\begin{proof}
From \eqref{eqelem}, we obtain that 
$$
\frac{1}{ \prod_{j=1}^m \prod_{k} (1-x_j y_k)  }
= \sum_\Lambda m_\lambda(x_1,\dots,x_m) h_\lambda(y)
=
\sum_{\pmb b} x^{\pmb b} h_{\pmb b}(y),
$$
which implies that
  $$
 \frac{1}{ \prod_{j,k} (1-x_j y_k)  } =  \frac{1}{ \prod_{j=1}^m \prod_{k} (1-x_j y_k)  }\cdot  \frac{1}{ \prod_{j\geq m+1} \prod_{k} (1-x_j y_k)  } =
\sum_{\pmb b} x^{\pmb b} h_{\pmb b}(y)  \sum_{\lambda} m_\lambda(x_{m+1},x_{m+2},\dots) h_\lambda(y).
 $$
Hence, using \eqref{eqeasy}, we get that
\begin{equation} \label{toprove}
  \frac{1}{\prod_{i=1}^m (1-q^{-1}x_i y_i)} \cdot \frac{1}{ \prod_{j,k} (1-x_j y_k)  } = \sum_{\pmb c} q^{-|\pmb c|} x^{\pmb c} y^{\pmb c}  \sum_{\pmb b} x^{\pmb b} h_{\pmb b}(y)  \sum_{\lambda} m_\lambda(x_{m+1},x_{m+2},\dots) h_\lambda(y).
 \end{equation}
Now, owing to  \cite{M}
$$
h_a[q^{-1}y+Y]= \sum_{c=0}^{a} q^{-c} y^{c} h_{a-c}[Y],
$$
we deduce that
 $$
\sum_{\pmb b,\pmb c} q^{-|\pmb c|} x^{\pmb c} y^{\pmb c}  x^{\pmb b} h_{\pmb b}(y) =
\sum_{\pmb a} \sum_{\pmb c \subseteq \pmb a} q^{-|\pmb c|} x^{\pmb a} y^{\pmb c}  h_{\pmb a-\pmb c}(y) = \sum_{\pmb a}  x^{\pmb a}  h_{a_1}[q^{-1}y_1+Y] \cdots h_{a_m}[q^{-1}y_m+Y].  
$$
Inserting the previous equation in the right-hand-side of \eqref{toprove}, we finally get
$$
 \frac{1}{\prod_{i=1}^m (1-q^{-1}x_i y_i)} \cdot \frac{1}{ \prod_{j,k} (1-x_j y_k)  } = \sum_{\pmb a}  x^{\pmb a}  h_{a_1}[q^{-1}y_1+Y] \cdots h_{a_m}[q^{-1}y_m+Y]    \sum_{\lambda} m_\lambda(x_{m+1},x_{m+2},\dots) h_\lambda(y),
$$
which proves the proposition.
\end{proof}  
We now obtain the limit as $t \to 1$ of the modified version of the $m$-symmetric Macdonald polynomials.  Recall that the integral form $J_\Lambda(x;q,t)$ of the $m$-symmetric Macdonald polynomials was introduced in \eqref{intform}.
\begin{proposition} \label{propJ1}
  For $\Lambda=(a_1,\dots, a_m;\lambda)$, let
$$
(q;q)_\Lambda = (q;q)_{a_1} \cdots  (q;q)_{a_m} (q;q)_{\lambda_1} \cdots (q;q)_{\lambda_{\ell(\lambda)}},
$$
where $(q;q)_r=(1-q)(1-q^2)\cdots (1-q^r)$.  We have that
$$
\lim_{t \to 1} J_\Lambda\left[\frac{X}{1-t};q,t\right]=
q^{|\pmb a|}(q;q)_\Lambda h_\Lambda\left[ \frac{X}{1-q};q\right]
=(q;q)_\Lambda
 h_{a_1}\left[x_1+\frac{qX}{1-q} \right]\cdots  h_{a_m}\left[x_m+\frac{qX}{1-q} \right]  h_\lambda\left[\frac{X}{1-q}\right] .
$$
\end{proposition}  
\begin{proof}
Define yet another scalar product as  
\begin{equation} \label{eqscal1}
\langle p_\Lambda(x;t)\, ,\, p_\Omega(x;t)  \rangle_{q,t}' =\delta_{\Lambda \Omega} \, q^{|\pmb a|}t^{{\rm Inv} (\pmb a)} z_\lambda
\end{equation}
and observe that
\begin{equation} \label{eqlimit}
\lim_{t \to 1} \,  \langle p_\Lambda(x;t)\, ,\, p_\Omega(x;t)  \rangle_{q,t}' =\delta_{\Lambda \Omega} \, q^{|\pmb a|} z_\lambda =   \langle p_\Lambda(x)\, ,\, p_\Omega(x)  \rangle'
 =  \langle \lim_{t \to 1} p_\Lambda(x;t)\, ,\, \lim_{t \to 1} p_\Omega(x;t)  \rangle',
 \end{equation}
 where the scalar product $ \langle \cdot \, ,\, \cdot  \rangle'$ was defined in
\eqref{eqscalq}. It is also immediate that
$$
\left \langle p_\Lambda[X;t]\, ,\, p_\Omega\left[\frac{X(1-q)}{(1-t)} ; t\right]  \right \rangle_{q,t}'= \delta_{\Lambda \Omega} \, q^{|\pmb a|}t^{{\rm Inv} (\pmb a)} z_\lambda(q,t) = \left \langle p_\Lambda(x;t)\, ,\, p_\Omega(x;t)  \right \rangle_{q,t} .
$$
Therefore, from $J_\Lambda(x;q,t)=c_\Lambda(q,t) P_\Lambda(x;q,t)$ and \eqref{squared}, we get
\begin{align*}
\left \langle P_\Lambda[X;q,t]\, ,\, J_\Omega\left[\frac{X(1-q)}{(1-t)} ;q, t\right]  \right \rangle_{q,t}'= \left \langle P_\Lambda(x;t)\, ,\, J_\Omega(x;t)  \right \rangle_{q,t} 
& =\delta_{\Lambda \Omega} c_{\Lambda}(q,t)  q^{|\pmb a|}t^{{\rm Inv} (\pmb a)} \prod_{s \in \Lambda} \frac{1 -q^{\tilde a(s)+1}t^{\tilde \ell(s)}}{1 -q^{a(s)}t^{\ell(s)+1}} \\
& = \delta_{\Lambda \Omega}  q^{|\pmb a|}t^{{\rm Inv} (\pmb a)} \prod_{s \in \Lambda} (1 -q^{\tilde a(s)+1}t^{\tilde \ell(s)}).
\end{align*}
Using $\lim_{t \to 1} P_{\Lambda}(x;q,t)=m_\Lambda(x)$ \cite{CL}, we thus deduce from \eqref{eqlimit} that
\begin{align*}
&  \left \langle m_\Lambda[X]\, ,\, \lim_{t \to 1} J_\Omega\left[\frac{X(1-q)}{(1-t)} ;q, t\right]  \right \rangle' \\
& \qquad \qquad  
=\lim_{t \to 1} \left \langle P_\Lambda[X;q,t]\, ,\, J_\Omega\left[\frac{X(1-q)}{(1-t)} ;q, t\right]  \right \rangle_{q,t}'= \delta_{\Lambda \Omega}  q^{|\pmb a|} \prod_{s \in \Lambda}   (1 -q^{\tilde a(s)+1}) = \delta_{\Lambda \Omega}  q^{|\pmb a|} (q;q)_{\Lambda} .
\end{align*}
Hence, from Proposition~\ref{propmhq}, we have that
$$
\lim_{t \to 1} J_\Lambda\left[\frac{X(1-q)}{(1-t)} ;q, t\right] =  q^{|\pmb a|} (q;q)_{\Lambda}  h_\Lambda(x;q),
$$
which is equivalent to
$$
\lim_{t \to 1} J_\Lambda\left[\frac{X}{1-t} ;q, t\right] =
q^{|\pmb a|}(q;q)_\Lambda h_\Lambda\left[ \frac{X}{1-q};q\right].
$$
\end{proof}

The previous proposition implies an important symmetry property of the coefficients $K_{\Omega \Lambda}(q,t)$ at $t=1$.
\begin{proposition} \label{propsym}
  The Kostka coefficients $K_{\Omega \Lambda}(q,1)$ are such that
  \begin{equation}
K_{\Omega \Lambda}(q,1) = K_{\sigma(\Omega) \sigma(\Lambda)}(q,1) 
  \end{equation}  
for any $\sigma \in S_m$, where $\sigma(\Lambda)=(\sigma(\pmb a);\lambda)=(a_{\sigma^{-1}(1)},\dots,a_{\sigma^{-1}(m)};\lambda)$.
    \end{proposition}  
\begin{proof}
  From \eqref{Kostka} and Proposition~\ref{propJ1}, we have at $t=1$ that
 \begin{equation} \label{eqori}
q^{|\pmb a|}(q;q)_\Lambda h_\Lambda\left[ \frac{X}{1-q};q\right] = 
\sum_\Omega K_{\Omega \Lambda}(q,1) s_\Omega(x;1),
\end{equation}
which implies that
\begin{equation} \label{eqoris}
q^{|\pmb a|}(q;q)_\Lambda h_{\sigma(\Lambda)}\left[ \frac{X}{1-q};q\right] = \sum_\Omega K_{\sigma(\Omega) \sigma(\Lambda)}(q,1) s_{\sigma(\Omega)}(x;1)
\end{equation}
for any $\sigma \in S_m$.
Let $K_\sigma$ be such that $K_\sigma f(x_1,\dots,x_m)=f(x_{\sigma(1)},\dots,x_{\sigma(m)})$ for any $\sigma \in S_m$.  We have
$$
K_\sigma  h_\Lambda\left[\frac{X}{1-q};q \right]=
 h_{a_1}\left[x_{\sigma(1)}+\frac{qX}{1-q} \right]\cdots  h_{a_m}\left[x_{\sigma(m)}+\frac{qX}{1-q} \right]  h_\lambda\left[\frac{X}{1-q}\right] =  h_{\sigma(\Lambda)}\left[\frac{X}{1-q};q \right]
 $$
 and $K_\sigma s_\Omega(x;1)= s_{\sigma(\Omega)}(x;1)$, where the last relation follows from \eqref{Tis} in the case $t=1$ (in which case $T_i=K_{i,i+1}$).
 Applying $K_\sigma$ on both sides of \eqref{eqori} thus yields 
  $$
q^{|\pmb a|}(q;q)_\Lambda h_{\sigma(\Lambda)}\left[ \frac{X}{1-q};q\right] = 
\sum_\Omega K_{\Omega \Lambda}(q,1) s_{\sigma(\Omega)}(x;1).
$$
Comparing the last equation with \eqref{eqoris}, it is immediate that $K_{\Omega \Lambda}(q,1) = K_{\sigma(\Omega) \sigma(\Lambda)}(q,1)$.
\end{proof}

\section{The coefficients $K_{\Omega \Lambda}(q,t)$ at $t=1$.} \label{sect8}

In this section, we will give a combinatorial interpretation for the coefficients  $K_{\Omega \Lambda}(q,1)$.   This combinatorial
interpretation will correspond to the major index statistic
of certain standard tableaux which will be constructed using Jeu de Taquin operations \cite{Fu}.

Let $T$ be a tableau with a hole in it. An inside Jeu de Taquin move consists in moving the hole upward by one unit or leftward by one unit depending on whether the entry above the hole is weakly larger or strictly smaller than the entry to its left.  In doing so, the entry where the hole now lies moves in the original position of the hole.

 Similarly, an outside Jeu de Taquin move consists in moving the hole downward by one unit or rightward by one unit depending on whether the entry below the hole is weakly smaller or strictly larger than the entry to its right.  In doing so, the entry where the hole now lies moves in the original position of the hole.

\begin{definition} \label{defJdt}
  Let $T$ be a (skew) tableau of $n$ letters.  For $a$ and $s$ such that  $1 \leq a<s \leq n$,  and such that there is only one occurrence of the letter $s$, we construct the tableau  $T_{s \to a}$ in the following way:
  \begin{enumerate}
  \item Create a hole in $T$ by deleting the letter $s$.
  \item Increase by one every letter in $T$  from $a$ to $s-1$ (so that the letters now go from $a+1$ to $s$).
  \item Use Jeu de Taquin moves to push the hole inside the letters $a+1$ to $s$. Let the resulting tableau be $T'$.
  \item Put letter $a$ in the position of the hole in $T'$ to form the (skew) tableau $T_{s \to a}$.
  \end{enumerate}  
\end{definition}

\begin{example}
\label{ExampleJdt}
    Let $T= \small{\tableau[scY]{ 1& 2 & 6& 10\\ 3& 8& 9\\ 4& 11& 13 \\ 5 & 12 \\ 7 } } $. We compute $T_{11 \to 2}$ as follows
\begin{enumerate}
    \item Delete the letter $s=11$:
$$
\small{\tableau[scY]{ 1& 2 & 6& 10\\ 3& 8& 9\\ 4& * & 13 \\ 5 & 12 \\ 7 } }
$$
    \item Increase by one the letters from $a=2$ to $s-1=10$:
$$
\small{\tableau[scY]{ 1& 3 & 7& 11\\ 4& 9& 10\\ 5& * & 13 \\ 6 & 12 \\ 8 } }
$$
    \item Use Jeu de Taquin to move the hole inside (but without moving it passed the letters smaller than $a+1=3$):
$$
\small{\tableau[scY]{ 1& 3 & 7& 11\\ 4& 9& 10\\ 5& * & 13 \\ 6 & 12 \\ 8 } } \quad \longrightarrow \quad  \small{\tableau[scY]{ 1& 3 & 7& 11\\ 4& *& 10\\ 5& 9 & 13 \\ 6 & 12 \\ 8 } }  \quad \longrightarrow \quad  T'=\small{\tableau[scY]{ 1& 3 & 7& 11\\ *& 4& 10\\ 5& 9 & 13 \\ 6 & 12 \\ 8 } }
$$
    \item Put the letter $a=2$ in the position of the hole:
$$
T_{11 \to 2}= \small{\tableau[scY]{ 1& 3 & 7& 11\\ 2& 4& 10\\ 5& 9 & 13 \\ 6 & 12 \\ 8 } }
$$    
\end{enumerate}
\end{example}

We can also define the reverse process.
\begin{definition} \label{defJdtrev}
  Let $T$ be a (skew) tableau of $n$ letters.  For $a$ and $s$ such that  $1\leq a < s \leq n$,  and such that there is only one occurrence of the letter $a$, we construct the tableau  $T_{a \to s}$ in the following way:
  \begin{enumerate}
  \item Create a hole in $T$ by deleting the letter $a$.
  \item Decrease by one every letter in $T$  from $a+1$ to $s$ (so that the letters now go from $a$ to $s-1$).
  \item Use Jeu de Taquin to push the hole outside the letters $a$ to $s-1$. Let the resulting tableau be $T'$.
  \item Put letter $s$ in the position of the hole in $T'$ to form the (skew) tableau $T_{a \to s}$.
  \end{enumerate}  
\end{definition}
We will also consider that $T_{a \to a}=T$.
From the elementary properties of the Jeu de Taquin \cite{Fu}, it is immediate that
$(T_{a \to s})_{s \to a}=(T_{s \to a})_{a \to s}=T$.  

\begin{example}
\label{ExampleJdtInv}
    In order to illustrate the reverse process, we start with the tableau $T$ obtained in  step $(4)$ of Example~\ref{ExampleJdt}, and compute $T_{2 \to 11}$. 

\begin{enumerate}
    
    \item Delete the letter $a=2$:
$$
\small{\tableau[scY]{ 1& 3 & 7& 11\\ *& 4& 10\\ 5& 9 & 13 \\ 6 & 12 \\ 8 } }
$$
\item Decrease by 1 every letter from $a+1=3$ to $s=11$:
$$
\small{\tableau[scY]{ 1& 2 & 6& 10\\ *& 3& 9\\ 4& 8 & 13 \\ 5 & 12 \\ 7 } }
$$
\item Use Jeu de Taquin to move the hole outside (but without moving it passed the letters larger than $s-1=10$):
$$
\small{\tableau[scY]{ 1& 2 & 6& 10\\ *& 3& 9\\ 4& 8 & 13 \\ 5 & 12 \\ 7 } } \quad \longrightarrow \quad   \small{\tableau[scY]{ 1& 2 & 6& 10\\ 3& *& 9\\ 4& 8 & 13 \\ 5 & 12 \\ 7 } } \quad \longrightarrow  \quad  T'= \small{\tableau[scY]{ 1& 2 & 6& 10\\ 3& 8& 9\\ 4& * & 13 \\ 5 & 12 \\ 7 } }
$$
\item Put the letter  $s=11$ in the position of the hole:
$$
T_{2 \to 11} = \small{\tableau[scY]{ 1& 2 & 6& 10\\ 3& 8& 9\\ 4& 11 & 13 \\ 5 & 12 \\ 7 } }
$$
\end{enumerate}
Observe that $T_{2 \to 11}$ is the original standard tableau in Example~\ref{ExampleJdt}.
\end{example}

Let $a$ and $b=a+1$ be consecutive letters in a standard tableau $T$.
The relative order between $a$ and $b$ is established in the following way.
We say that
$b$ is southwest of $a$ if $b$ lies in a column weakly to the left of that of $a$. Similarly, we say that  $b$ is northeast of $a$ if $b$ lies in a column strictly to the right of that of $a$.

It will prove important later in this section that the Jeu de Taquin operation
$T_{a \to s}$   
described earlier preserves relative orders.  Observe that $T_{s \to a}$ will also preserve relative orders since it is essentially the inverse of $T_{a \to s}$.  
\begin{lemma} \label{lemmaOrder1} Let $T$ be a standard tableau of size $n$,
  and let  $T' = T_{a \to s}$.
  For any consecutive letters $b$ and $b+1$ such that $a < b < s$, we have
  that the relative order of $b$ and $b+1$ in $T$ is the same as the relative order of $b-1$ and $b$ in $T'$.
  \end{lemma}  
\begin{proof}
The relative order between the letters will change only if the relative order of the letters $b$ and $b+1$ changes when performing the Jeu de Taquin in Step (3) of 
Definition~\ref{defJdt}.
Since the Jeu de Taquin can only move the letters one cell to the right or one cell to the left, the only possibility for the relative orders to change is if the Jeu de Taquin changes $$
{\footnotesize \tableau*[mcY]{
b  \cr
b+1
}} \qquad {\rm into} \qquad   {\footnotesize \tableau*[mcY]{
b  
 & b+1
}} 
$$     
or vice versa.  In order for the first case to occur, we would need to have the following situation
$$
{ \footnotesize \tableau*[mcY]{
a & b  \cr
c &  b+1  
}}
\quad \longrightarrow \quad 
{ \footnotesize
\tableau*[mcY]{
\tf & b  \cr
c & b+1  
}
}
\quad \longrightarrow \quad 
{ \footnotesize
\tableau*[mcY]{
b &  \tf  \cr
c & b+1
}}
\quad \longrightarrow\quad 
{ \footnotesize
\tableau*[mcY]{
b & b+1 \cr
c & \tf  
}}
$$
which, by definition of Jeu de Taquin,  leads to the contradiction $b<c<b+1$. The case where the Jeu de Taquin changes
$$
  {\footnotesize \tableau*[mcY]{
b  
 & b+1
}} 
\qquad {\rm into} \qquad
{\footnotesize \tableau*[mcY]{
b  \cr
b+1
}} 
$$     
can be proved in a similar manner.
\end{proof}  

\begin{example}
  Using 
  $$
  T= \small{\tableau[scY]{ 1& 3 & 7& 11\\ 2& 4& 10\\ 5& 9 & 13 \\ 6 & 12 \\ 8 } }
  \qquad
  {\rm and} \qquad    
   T_{2 \to 11} = \small{\tableau[scY]{ 1& 2 & 6& 10\\ 3& 8& 9\\ 4& 11 & 13 \\ 5 & 12 \\ 7 } }
$$ 
    computed in Example~\ref{ExampleJdtInv}, 
    we see that the pair $3,4$ in $T$ has the same relative order as the pair $2,3$ in $T_{2 \to 11}$.
The relative order is also preserved when comparing the pair $a,a+1$ in $T$ with the pair $a-1,a$ in $T_{2 \to 11}$ for $a=4,\dots, 10$.
\end{example}

\begin{lemma}  \label{lemmaOrder} Let $T$ be a standard tableau of size $n$.  For any consecutive letters
  $a$ and $a+1$  and any $a+1 < s \leq n$, we have that
  the relative order of the letters  $a$ and $a+1$ in $T$ is the same as the relative order of the letters $s-1$ and $s$ in the tableau $T'$, where
  $$
  T' =  (T_{a+1 \to s})_{a \to s-1}.
    $$
\end{lemma}

\begin{proof}
  Let $a$ and $b=a+1$ be in positions $(i,j)$ and $(k, \ell)$ respectively. We observe  that if $k>i$ then $\ell \leq j$ since there would otherwise need to be an
  entry $c$ such that $a<c<b$ in position  $(k, j)$, which is impossible given that
  $a$ and $b=a+1$ are consecutive.
$$
{\footnotesize \tableau*[mcY]{
a & \ldots & c \cr
\vdots & \ddots & \vdots \cr
d & \ldots  & b
}}
$$   
Hence, if $b$ is strictly to the right of $a$ then it must lie in the row of $a$ or in a higher one.

When going from $T$ to $T_{b \to s}$, the hole in the position of the
 letter $b$ will follow a path when pushed outside using Jeu de Taquin.
Let $P_{b}$ be this path. Similarly, we will let $P_a$ be the path followed by the hole in the position of the letter $a$ when going from 
 $T_{b \to s}$ to $T'=(T_{b \to s})_{a \to s-1}$. Observe that in $T'$, the ending points of $P_a$ and $P_b$ are occupied respectively by the letters $s-1$ and $s$. Hence, $P_a$ cannot end southeast of $P_b$.

We will consider separately the two following cases:

\begin{enumerate}
    \item $b$ lies in a column strictly to the right of the column in which  $a$ lies ($b$ is northeast of $a$).

    \item $b$ lies in a column weakly to the left of the column in which $a$ lies ($b$ is southwest of $a$).
\end{enumerate}

We first consider Case (1).
In this case, given that $P_a$ starts weakly below $P_b$ and cannot end southeast of $P_b$, the relative order of the letters can only change if $P_{a}$ goes through a vertical segment of $P_b$.  We will now see that this is impossible, as illustrated in Figure~\ref{fig:NoV}, given that $P_a$ cannot even touch $P_b$ on a vertical segment.
\begin{centering}
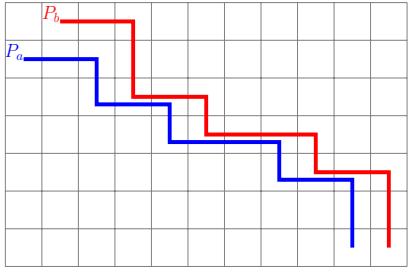
\begin{figure}
\resizebox{5.5cm}{3.5cm}{
  \begin{tikzpicture} 
\node[thick, font=\fontsize{14}{0}\selectfont, thick] at (1.25,6.7)  {$\color{red}{ P_{b} } $};
\node[thick, font=\fontsize{14}{0}\selectfont, thick] at (0.25,5.7){$\large \color{blue}{P_{a}}$};
 \draw[step=1cm, color=gray] (0, 0) grid (11, 7);
\draw[line width=3pt, red] (1.5,6.5) -- (3.5,6.5) -- (3.5,4.5) -- (5.5,4.5) -- (5.5,3.5) -- (8.5,3.5) -- (8.5, 2.5) -- (10.5, 2.5) -- (10.5, 0.5);
\draw[line width=3pt, blue] (0.5,5.5) -- (2.5,5.5) -- (2.5,4.3) -- (4.5,4.3) -- (4.5,3.3) -- (7.5,3.3) -- (7.5, 2.3) -- (9.5, 2.3) -- (9.5, 0.5);
\end{tikzpicture}
}
\caption{$P_{a}$ cannot touch $P_{b}$ on a vertical segment.}
\label{fig:NoV}
\end{figure}
\end{centering}

Suppose that $P_{a}$ approaches a vertical segment of $P_{b}$.
Let $c_{1}, c_{2}$ be the labels in $T$ of the squares occupied by the vertical segment.  After the vertical Jeu de Taquin, the squares become respectively $c_2$ and $c_3$. We then use an overline to denote their values in $T_{b \to s}$,
with $\bar c_2=c_2-1$, $\bar c_3=c_3-1$ or $\bar c_3=s$, $\bar b_1= b_1-1$ or
$\bar b_1=b_1$ (depending on whether $b < b_1 < s$ or not), and
$\bar b_2= b_2-1$ or
$\bar b_2=b_2$.
$$
{ \footnotesize \tableau*[mcY]{
b_{1} & c_{1}  \cr
b_{2} & c_{2}  
}}
\longrightarrow
{ \footnotesize
\tableau*[mcY]{
b_{1} & c_{2}  \cr
b_{2} & c_{3}  
}
}
\longrightarrow
{ \footnotesize
\tableau*[mcY]{
\overline{b}_{1} & \overline{c}_{2}  \cr
\overline{b}_{2} & \overline{c}_{3}  
}}
$$

When, in computing $T'=(T_{b \to s})_{a \to s-1}$, the hole corresponding to $a$
(represented by an $a$ in the following diagram)  arrives in the position in which ${b}_{1}$ is located, we have that $a < b_1 <b_2 <c_2 \leq s$, which implies that
$\bar b_1=b_1-1$ and $\bar b_2 = b_2-1$. Therefore, 
we have that $\overline{b}_{2} = b_{2} - 1 < c_{2} -1 = \overline{c}_{2}$ and the Jeu de Taquin move is the following: 
$$
{ \footnotesize\tableau*[mcY]{
a & \overline{c}_{2}  \cr
\overline{b}_{2} & \overline{c}_{3}  
}}
\to
{ \footnotesize \tableau*[mcY]{
\overline{b}_{2} & \overline{c}_{2}  \cr
a & \overline{c}_{3}  
}}
$$
We have thus shown that the path $P_{a}$ cannot touch the path $P_b$ on a vertical segment, which implies that in $T'$ the letter 
$s$ lies in a column strictly to the right of the letter $s-1$.

We now consider Case (2) in which $b=a+1$ lies in column weakly to the left of that of $a$. This case is similar to the previous one, with the path $P_{b}$ never going through a horizontal segment of the path $P_{a}$.
Geometrically we can think of this case as the conjugate of the previous one, which we illustrate in Figure~\ref{fig:NoH}. We omit the proof  which is exactly as in the previous case.

\begin{centering}
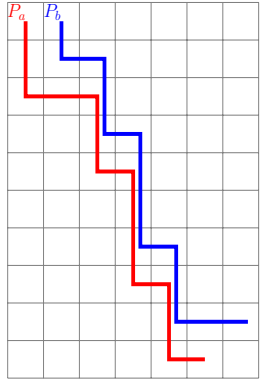
\begin{figure}
\resizebox{3.5cm}{5cm}{
\begin{tikzpicture}
\node[thick, font=\fontsize{14}{0}\selectfont, thick] at (0.25,9.75)  {$\color{red}{ P_{a} } $};
\node[thick, font=\fontsize{14}{0}\selectfont, thick] at (1.25,9.75){$\large \color{blue}{P_{b}}$};
 \draw[step=1cm, color=gray] (0, 0) grid (7, 10);
\draw[line width=3pt, red] (0.5,9.5) -- (0.5,7.5) -- (2.5,7.5) -- (2.5,5.5) -- (3.5,5.5) -- (3.5,2.5) -- (4.5, 2.5) -- (4.5, 0.5) -- (5.5, 0.5);
\draw[line width=3pt, blue] (1.5,9.5) -- (1.5,8.5) -- (2.7,8.5) -- (2.7,6.5) -- (3.7,6.5) -- (3.7,3.5) -- (4.7, 3.5) -- (4.7, 2.5) -- (4.7,1.5) -- (6.7, 1.5);
\end{tikzpicture}
}
\caption{$P_{b}$ cannot touch $P_{a}$ on a horizontal segment.}
\label{fig:NoH}
\end{figure}
\end{centering}
\end{proof}

\begin{example}
Consider the tableau $(T_{3 \to 15})_{2 \to 14}$
obtained from the following tableau $T$:
\begin{equation}
T={\small{\tableau[scY]{ 1& \color{red}{3}& 5& 6& 11\\2& \color{red}{4}& \color{red}{7}& \color{red}{9} \\8& 10& 13& \color{red}{15}\\12& 14 } } }
\quad \longrightarrow \quad T_{3 \to 15}=
{\small{\tableau[scY]{ 1& 3& 4& 5& 10\\ \color{blue}{2}& \color{blue}{6}& \color{blue}{8}& 14\\7& 9& \color{blue}{12} & 15\\11& 13 } } }
\quad \longrightarrow \quad   T'=(T_{3 \to 15})_{2 \to 14}=
{\small{\tableau[scY]{ 1&2&3&4&9\\ 5&7&11&13 \\ 6&8&14&15 \\ 10&12 } } }
\end{equation}
where the paths $P_3$ and $P_2$ are indicated in red and blue respectively. Note that the relative order of the letters $2$ and $3$ in $T$ is the same as the relative order of the letters $14$ and $15$ in  $(T_{3 \to 15})_{2 \to 14}$.

\end{example}

\begin{definition} Let $\alpha=(\alpha_1,\dots,\alpha_\ell)$ be a vector of nonnegative integers, and let $T$ be a standard tableau of size $|\alpha|$.  
 For $i \in \{1,\dots,\ell\}$ and $c=\alpha_1+\cdots+\alpha_{i-1}+1$, we let 
  $$
  {\rm maj}_{\alpha,i}(T) = \sum_{j : _{c+j}\swarrow^{c+j-1}} j ,
$$
  where the sum is over all $j \in \{1,\dots,\alpha_i-1 \}$ such that the letter $c+j$ is southwest of the letter $c+j-1$ in $T$.
    We then define the major index of $T$ relative to $\alpha$ as
  $$
 {\rm maj}_{\alpha}(T) = \sum_{i=1}^\ell  {\rm maj}_{\alpha,i}(T).
  $$
\end{definition}  

\begin{example}  Let $\alpha = (3,1,4,5,2)$ and suppose that 
$$
T=
{\small{\tableau[scY]{ 1& 2 & 5 & 10 & 15 \\ 3&6&8&11\\ 4&9&12&14 \\ 7&13 
}}}
$$
To $\alpha$ corresponds the following auxiliary diagram whose row lengths are the entries of $\alpha$:
$$
{\small{\tableau[scY]{1^{\color{red}{0}}&2^{\color{red}{1}}&3^{\color{red}{2}} \\ 4^{\color{red}{0}} \\ 5^{\color{red}{0}}&6^{\color{red}{1}}&7^{\color{red}{2}}&8^{\color{red}{3}} \\ 9^{\color{red}{0}}&10^{\color{red}{1}}&11^{\color{red}{2}}&12^{\color{red}{3}}&13^{\color{red}{4}}\\ 14^{\color{red}{0}}&15^{\color{red}{1}} 
}}}
$$
If $a$ and $a+1$ are in row $i$ of the auxiliary diagram, and $a+1$ is southwest of  $a$ in $T$, then we add the superscript of $a+1$ to ${\rm maj}_{\alpha,i}(T)$. We thus have
$$
{\rm maj}_{\alpha,1}(T)=2 , \quad  {\rm maj}_{\alpha,2}(T)=0,\quad
 {\rm maj}_{\alpha,3}(T)=1+2,\quad  {\rm maj}_{\alpha,4}(T)=2+3+4,\quad  {\rm maj}_{\alpha,5}(T)=0  
$$
which implies that $ {\rm maj}_{\alpha}(T)=2+0+3+9+0=14$.
\end{example}

Let $T$ be a standard tableau in $n$ letters. For  an interval $[a,b]$,
with $1 \leq a <b < n$ and an $s >  b$, we will let
$$
T_{[a,b] \to s} = \bigl(\cdots (T_{b \to s})_{b-1 \to s-1} \cdots\bigr)_{a \to s+a-b}.
$$

Given an $m$-partition $\Omega$ such that $|\Omega|=n$, we will say that $T$ is a standard filling of $\Omega$ if the cells of $\Omega$ (not including the circles) are filled with the integers $1,\dots,n$ in a standard way.  For instance, 
a possible standard filling of  $\Omega=(2,0,1; 3,2 )$ is 
 $$
        {\tableau[scY]{ 1&2 &4 & \bl \\ 3 &6  & \bl \cercle{1} \\5 & 7& \bl  \\ 8 & \bl \cercle{3} \\ \bl \cercle{2} }}
$$

\begin{definition} \label{defTa}
 Let $T$ be  a standard filling of the $m$-partition $\Omega$, with
  $n=|\Omega|$. We first
 let $T_\circ$ be the tableau obtained from $T$ by changing, for $i$ from 1 to $m$, the $i$-circle into a cell filled with the letter $n+i$. For $\pmb a =(a_1,\dots,a_m)$ such that $|\pmb a| \leq n$, we will let $T_{\pmb a}$ be the standard tableau obtained from $T_\circ$ in the following way.
  Let  $T^{(m+1)}=T_\circ$, and then define recursively
  $$
  T^{(i)}=T^{(i+1)}_{[n+1-a_{i}-\cdots- a_m,n-a_{i+1}-\cdots -a_m] \to  n+i-1-a_{i+1}-\cdots -a_m}
  $$
 for $i=m,\dots,1$.
 We then let $T_{\pmb a}=T^{(1)}$ (which is also equal to $T^{(2)}$ by construction).  If the letters in
 $T_\circ$ are divided into the blocks $(n-|\pmb a|,a_1,\dots,a_m,1^m)$ , then those in $T^{(m)}$ are divided into the blocks $(n-|\pmb a|,a_1,\dots,a_{m-1},1^{m-1},a_m,1)$ (the letters corresponding to $a_m$ have been shifted by $m-1)$. Using this process again and again, we get that the letters in $T_{\pmb a}$
  are divided into the blocks $(n-|\pmb a|,a_1,1,a_2,1,\dots,a_m,1)$.
\end{definition}  

\begin{example} \label{ex27} We will compute $T_{\pmb a}$ for $\pmb a=(2,1,3)$ and 
$$
T ={\small{\tableau[scY]{ 1&2&4&8\\ 3&5&6 &\bl \tcercle{\rm{1}}\\ 7 & 9 & \bl \tcercle{\rm{2}} \\ \bl \tcercle{\rm{3}}
}}}
$$
First, after replacing the 1-circle, 2-circle and 3-circle by cells filled respectively by the letters $10,11$ and $12$, we obtain
$$
 T_{\circ}= T^{(4)} ={\small{\tableau[scY]{ 1&2&4&8\\ 3&5&6 & 10 \\ 7 & 9 & 11 \\ 12
}}}
$$
whose blocks are $(\mathbf{3},2,1,3, \mathbf{1}, \mathbf{1}, \mathbf{1})$ with the $\mathbf{1}$'s stemming from the circles and the $\mathbf{3}$ corresponding to $n-|\pmb a|=3$. In order to go from $(\mathbf{3},2,1,3, \mathbf{1}, \mathbf{1}, \mathbf{1})$ to  $(\mathbf{3},2,1, \mathbf{1}, \mathbf{1}, 3, \mathbf{1})$, we need to compute $T^{(4)}_{[7,9] \to 11}$:
$$
 T^{(4)}_{9\to 11} = {\small{\tableau[scY]{ 1&2&4&8\\ 3&5&6 & 9 \\ 7 & 10 & 11 \\ 12}}} \quad \longrightarrow \quad T^{(4)}_{[8,9] \to 11} = {\small{\tableau[scY]{ 1&2&4&8\\ 3&5&6 & 10 \\ 7 & 9 & 11 \\ 12}}} \quad \longrightarrow \quad T^{(3)} = T^{(4)}_{[7,9] \to 11} = {\small{\tableau[scY]{ 1&2&4&6\\ 3&5&7 & 10 \\ 8 & 9 & 11 \\ 12}}} 
$$
 Then, to go from $(\mathbf{3},2,1, \mathbf{1}, \mathbf{1}, 3, \mathbf{1})$ to $(\mathbf{3},2, \mathbf{1},1, \mathbf{1}, 3, \mathbf{1})$ we need to compute $T^{(3)}_{[6,6] \to 7}$.
$$
T^{(2)} = T^{(3)}_{[6,6] \to 7} = T^{(3)}_{6 \to 7}  = {\small{\tableau[scY]{ 1&2&4&7\\ 3&5&6 & 10 \\ 8 & 9 & 11 \\ 12}}} 
$$
Finally, as we don't need to move $a_{1}=2$ since it already has a $\mathbf{1}$  to its right, we have
$$
T_{\pmb a} = T^{(1)} =  T^{(2)}
$$
\end{example}

We will now provide a combinatorial interpretation for the expansion of the $m$-symmetric Macdonald polynomials in terms of $m$-symmetric Schur functions.

\begin{theorem} \label{theo}
  Let $\Lambda=(\pmb a;\lambda)$ be an arbitrary $m$-partition.
The $(q,t)$-Kostka coefficients at $t=1$ in 
  $$
\lim_{t \to 1 } J_\Lambda \left [ \frac{X}{1-t};q,t\right]=  \sum_{\Omega} K_{\Omega \Lambda}(q,1) s_{\Omega}(x;1)
$$
have the following combinatorial interpretation.  If 
$\Omega$ is dominant, then
\begin{equation} \label{kostkadom}
K_{\Omega \Lambda}(q,1)=\sum_{T: \,{\rm sh}(T)=\Omega} q^{{\rm maj}_{\tilde \Lambda}(T_{\pmb a})}  ,
\end{equation}
where the sum is over all standard fillings of the $m$-partition  $\Omega$,
where
$$
\tilde \Lambda=(\lambda_1,\dots,\lambda_\ell,a_1+1,\dots,a_m+1),
$$
and where $T_{\pmb a}$ was introduced in Definition~\ref{defTa}.

If $\Omega$ is not dominant, then $\sigma(\Omega)$ is dominant for some $\sigma \in S_m$. In this case, using Proposition~\ref{propsym}, we can compute
$K_{\Omega \Lambda}(q,1)=K_{\sigma(\Omega) \sigma(\Lambda)}(q,t)$ using \eqref{kostkadom}.

The combinatorial interpretation for $K_{\Omega \Lambda}(q,1)$ implies in particular that, when $q=t=1$, we have 
$$
K_{\Omega \Lambda}(1,1) = \# \{ {\rm standard~fillings~of~} \Omega \}=  \# \{ {\rm standard~tableaux~of~shape~} \pmb b \cup \mu \} .
$$
\end{theorem}
We first illustrate Theorem~\ref{theo} before proceeding to its proof.
\begin{example}
    Let $\Lambda = (2,1,3;3)$, and $T$ be as in Example~\ref{ex27}. We have that ${\rm sh}(T)=(3,2,0;4)=\Omega$ is dominant. The contribution of $T$ to $K_{\Omega \Lambda}(q,1)$ is thus $q^{{\rm maj}_{\tilde \Lambda}(T_{\pmb a})}$, where
    $\pmb a=(2,1,3)$ and $\tilde \Lambda=(3,3,2,4)$. In Example~\ref{ex27}, we obtained
    $$
T_{\pmb a}  =  {\small{\tableau[scY]{ 1&2&4&7\\ 3&5&6 & 10 \\ 8 & 9 & 11 \\ 12}}} 
    $$
    Using the auxiliary diagram    
$$
{\small{\tableau[scY]{1^{\color{red}{0}}&2^{\color{red}{1}}&3^{\color{red}{2}}\\ 4^{\color{red}{0}} & 5^{\color{red}{1}} & 6^{\color{red}{2}} \\ 7^{\color{red}{0}}&8^{\color{red}{1}}  \\  9^{\color{red}{0}} & {\tiny 10^{\color{red}{1}}} & 11^{\color{red}{2}} & 12^{\color{red}{3}}
}}}
$$
we get that ${{\rm maj}_{\tilde \Lambda}(T_{\pmb a})}=2+1+1+2+3=9$ with contributions from the letters
$3,5,8,11,12$.
\end{example}

\begin{example}
    Let $\Lambda=(2,1,3;2,1)$ and $\Omega = (2,0,3;4)$.
In this case, $\Omega$ is not dominant.  Using the permutation $\sigma=[2,3,1]$, we have that $\sigma(\Omega)=(3,2,0;4)$ is now dominant. Hence,
$$
K_{\Omega \Lambda}(q;1)=K_{\sigma(\Omega) \sigma(\Lambda)}(q;1) = K_{(3,2,0;4) (3,2,1;2,1)}(q,1)
$$
Note that in this case $\Lambda'=\sigma(\Lambda)=(3,2,1;2,1)$ is dominant even though it does not have to be necessarily this way.
One of the  tableaux contributing to $K_{(3,2,0;4) (3,2,1;1)}(q,1)$ is
$$
T = {\small{\tableau[scY]{ 1&2&4&8\\ 3&5&6 &\bl \tcercle{\rm{1}}\\ 7 & 9 & \bl \tcercle{\rm{2}} \\ \bl \tcercle{\rm{3}}
}}}
$$
and its contribution will be $q^{{\rm maj}_{\tilde \Lambda'}(T_{\pmb a})}$, where
    $\pmb a=(3,2,1)$ and $\tilde \Lambda'=(2,1,4,3,2)$.
We first compute $T_{\pmb a}$. We start with
$$
T^{(4)}=T_{\circ} = {\small{\tableau[scY]{ 1&2&4&8\\ 3&5&6 & 10 \\ 7 & 9 & 11 \\ 12}}}
$$
whose blocks are $({\bf 3},3,2,1, \pmb 1, \pmb 1, \pmb 1)$. We then obtain $T^{(3)}=T^{(4)}_{9 \to 11}$ as follows:
$$
T^{(4)}={\small{\tableau[scY]{ 1&2&4&8\\ 3&5&6 & 10 \\ 7 & 9 & 11 \\ 12}}} \quad {\longrightarrow}\quad  T^{(4)}_{9\to 11}=
{\small{\tableau[scY]{ 1&2&4&8\\ 3&5&6 & 9 \\ 7 & 10 & 11 \\ 12}}} 
$$
whose blocks are now $({\pmb 3},3,2, \pmb 1, \pmb 1, 1, \pmb 1)$. In order to reach the blocks $({\pmb 3},3, \pmb 1, 2, \pmb 1, 1, \pmb 1)$, we then have to compute $T^{(2)}=T^{(3)}_{[7,8] \to 9}$:
$$
T^{(3)}= {\small{\tableau[scY]{ 1&2&4&8\\ 3&5&6 & 9 \\ 7 & 10 & 11 \\ 12}}}    \quad {\longrightarrow} \quad T^{(3)}_{8\to 9} =
{\small{\tableau[scY]{ 1&2&4&8\\ 3&5&6 & 9 \\ 7 & 10 & 11 \\ 12}}}  \quad {\longrightarrow} \quad   (T^{(3)}_{8\to 9})_{7 \to 8}=T^{(3)}_{[7,8] \to 9}=T^{(2)}=
{\small{\tableau[scY]{ 1&2&4&7\\ 3&5&6 & 9 \\ 8 & 10 & 11 \\ 12}}} 
$$
With $T_{\pmb a}=T^{(1)}=T^{(2)}$ and with 
the auxiliary diagram associated to $\tilde \Lambda'=(2,1,4,3,2)$ given by
$$
{\small{\tableau[scY]{
1^{\color{red}{0}}&2^{\color{red}{1}}  \\ 
3^{\color{red}{0}} \\
4^{\color{red}{0}}&5^{\color{red}{1}}&6^{\color{red}{2}}&7^{\color{red}{3}} \\ 8^{\color{red}{0}} & 9^{\color{red}{1}} & 10^{\color{red}{2}} \\  11^{\color{red}{0}} & 12^{\color{red}{1}} 
}}}
$$
we get that the contributions to ${\rm maj}_{\tilde \Lambda'}(T_{\pmb a})$  come from the letters  $5,10$ and $12$. Hence ${\rm maj}_{\tilde \Lambda'}(T_{\pmb a})= 1+2+1=4$ which implies that the contribution of $T$ to $K_{\Omega \Lambda} (q,1)$ is $q^4$.

\end{example}

\begin{proof}[Proof of Theorem~\ref{theo}]
From Proposition~\ref{propJ1}, we need to prove that, for a dominant $\Omega$, we have 
$$(q;q)_\Lambda
 h_{a_1}\left[x_1+\frac{qX}{1-q} \right]\cdots  h_{a_m}\left[x_m+\frac{qX}{1-q} \right]  h_\lambda\left[\frac{X}{1-q}\right] \Bigg |_{s_\Omega}=\sum_{T: \,{\rm sh}(T)=\Omega} q^{{\rm maj}_{\tilde \Lambda}(T_{\pmb a})}  
$$
or, equivalently, that
\begin{equation} \label{bigsum}
  (q;q)_\lambda  h_\lambda\left[\frac{X}{1-q}\right]
(q;q)_{a_1} h_{a_1}\left[x_1+\frac{qX}{1-q} \right]\cdots  (q;q)_{a_m}h_{a_m}\left[x_m+\frac{qX}{1-q} \right] \Bigg |_{s_\Omega}=\sum_{T: \,{\rm sh}(T)=\Omega} q^{{\rm maj}_{\tilde \Lambda}(T_{\pmb a})} . 
\end{equation}
We now use 
$$
h_{a_i}\left[x_i+\frac{qX}{1-q} \right] = \sum_{c_i=0}^{a_i} q^{a_i-c_i} h_{a_i-c_i}\left[\frac{X}{1-q} \right]x_i^{c_i}
$$
to rewrite the left-hand side of \eqref{bigsum} as
\begin{equation} \label{sumc}
\sum_{c_1=0}^{a_1} \dots \sum_{c_m=0}^{a_m} q^{|\pmb a|-|\pmb c|}
(q;q)_\alpha  h_\alpha\left[\frac{X}{1-q}\right]
\frac{(q;q)_{a_1}\cdots (q;q)_{a_m}}{(q;q)_{a_1-c_1} \cdots (q;q)_{a_m-c_m}} 
x^{\pmb c} \Bigg |_{s_\Omega},
\end{equation}
where $\alpha=(\lambda_1,\dots,\lambda_\ell,a_1-c_1,\dots,a_m-c_m)$. Now, it is known that \cite{M}
\begin{equation} \label{eqh}
(q;q)_\alpha  h_\alpha\left[\frac{X}{1-q}\right]= \sum_P q^{{\rm maj}_{\alpha}(P)}  s_{{\rm sh}(P)}(x) ,
\end{equation}
where the sum if over all standard tableaux $P$ of size $|\alpha|$.  Since $\Omega=(\pmb b;\mu)$ is dominant, we get from \eqref{eqkins} and Proposition~\ref{dualSD}  that 
\begin{equation} \label{eqS}
x^{\pmb c} s_{\nu}(x) \Big |_{s_\Omega}= k_\Gamma(x)\Big |_{s_\Omega}= D_{\Omega \Gamma} = \# \mathcal S_{\Omega \Gamma},
\end{equation}
where $\Gamma=(\pmb c;\nu)$. We recall that $Q \in \mathcal S_{\Omega \Gamma}$ if $Q$
is a filling of the skew shape $(\pmb b \cup  \mu)/\nu$ such that letter $i$ appears $c_i$ times and always within the first $b_i$ columns of $Q$.
For $i=1,\dots,m$, we then change the circle $i$ in $Q$ into a square filled with the letter $i$.
Note that this produces a skew tableau $Q_\circ$ by Lemma~\ref{lemmaskew}.
We now standardize $Q_\circ$ in the following way. Suppose that $r=|\nu|=|\alpha|$.
We replace all the letters $1$ in $Q$ (there are $c_1+1$ of them) by the letters $r+1,\dots,r+c_1+1$ ordered from left to right.  We replace all the letters $2$ in $Q$ (there are $c_2+1$ of them) by the letters $r+c_1+2,\dots,r+c_1+c_2+2$ ordered from left to right, and so on.  Let the resulting skew tableau be $Q_\circ^{\rm st}$.  Letting $\Gamma=(\pmb c; {\rm sh}(P))$, we thus have that
$$
(q;q)_\alpha  h_\alpha\left[\frac{X}{1-q}\right] x^{\pmb c} \bigg |_{s_\Omega}=  \sum_{P} \sum_{Q \in  \mathcal S_{\Omega \Gamma} } q^{{\rm maj}_{\tilde \alpha}(T_{P,Q})}, 
$$
where $T_{P,Q}$ is the standard tableau obtained from the union of the tableau $P$ and
 the skew tableau $Q_\circ^{\rm st}$ described earlier, and
where $\tilde \alpha=(\lambda_1,\dots,\lambda_\ell,a_1-c_1,\dots,a_m-c_m,1^{c_1+1},\dots,1^{c_m+1})$.

We now let
$$
T_{P,Q}'=(T_{P,Q})_{[r-a_m+c_m+1,r] \to n+m-c_m-1},
$$
where we recall that $r=|\alpha|=|\lambda|+|\pmb a|-|\pmb c|$, and where $n=|\Omega|$.  From Lemma~\ref{lemmaOrder}, we have that
$$
{\rm maj}_{\tilde \alpha}(T_{P,Q})={\rm maj}_{\tilde \alpha'}(T_{P,Q}'),
$$
where
$$
\tilde \alpha'=(\lambda_1,\dots,\lambda_\ell,a_1-c_1,\dots,a_{m-1}-c_{m-1},1^{c_1+1},\dots,1^{c_{m-1}+1},a_m-c_m,1^{c_m+1}).
$$

Now, fix a $T'=T_{P,Q}$.  This $T'$ can originate from many pairs $(P,Q)$, each pair corresponding to a certain $c_m$.
Since we only focus on the last $a_m+1$ letters of $T'$, we will refer to those letters as $1,\dots,a_m+1$ for simplicity. Let $s \geq 0$ be such that in $T'$ the letters 
$s+1,s+2,\dots,a_m+1$ are ordered from left to right and such that this sequence is the longest (that is, either $s=0$ or the letter $s+1$ is southwest of the letter $s$ in $T'$).  Note that, for a given $(P,Q)$, this 
can only occur if the corresponding $c_m$ is such that 
$c_m +1 \leq a_m-s+1$ given that in that case the last
$c_m+1$ letters
of $T'$ are ordered from left to right by construction (they correspond to the occurrences of the letter $m$ in $Q_\circ$ ordered form left to right). 
We will see that for such a $T'$, we have
\begin{equation} \label{eqlast}
 \sum_{c_m=0}^{a_m-s}  q^{a_m-c_m} \frac{ (q;q)_{a_m}}{ (q;q)_{a_m-c_m}}
q^{{\rm maj}_{\tilde \alpha'}(T')} =  q^{{\rm maj}_{\gamma^{(m)}}(T')},
\end{equation}
where
$$
\gamma^{(m)}=(\lambda_1,\dots,\lambda_\ell,a_1-c_1,\dots,a_{m-1}-c_{m-1},1^{c_1+1},\dots,1^{c_{m-1}+1},a_m+1).
$$
First suppose that $s\neq 0$.  
By definition of  ${\rm maj}_{\tilde \alpha'}$ and 
${\rm maj}_{\gamma^{(m)}}$, the letters smaller than $s+1$ will have the same contribution in both statistics (given that $s \leq a_m-c_m$) while the letters larger than $s+1$ will contribute 0 to both statistics given that by definition those letters are ordered from left to right. Hence, the only difference  between the two statistics is in the contribution of the letter $s+1$.
By definition of  ${\rm maj}_{\tilde \alpha'}$,
the letter $s+1$ will contribute $s$ to ${\rm maj}_{\tilde \alpha'}$ if
$s < a_m-c_m$, and 0 otherwise.  In the sum, we thus have that the 
  letter $s+1$ will  contribute $s$ to the ${\rm maj}_{\tilde \alpha'}$ statistic in   $T'$ if $s \neq a_m-c_m$ while it will contribute 0 otherwise.  We thus have to show that
\begin{equation} \label{eqsumT}
q^s \frac{(q;q)_{a_m}}{(q;q)_{s}} +  q^s \sum_{c_m=0}^{a_m-s-1} q^{a_m-c_m}  \frac{ (q;q)_{a_m}}{ (q;q)_{a_m-c_m}}
  =  q^s
\end{equation}
since on the right-hand side of \eqref{eqlast} the contribution to the ${\rm maj}_{\gamma^{(m)}}$ statistic of the letter $s+1$ in $T'$  is $q^s$.
The previous identity amounts to 
 \begin{equation} \label{eqsumT2}
 \frac{(q;q)_{a}}{(q;q)_{s}}= 1 -   \sum_{i=s+1}^{a} q^{i}  \frac{ (q;q)_{a}}{ (q;q)_{i}}.
 \end{equation}
 For a fixed $a$, the identity can be easily seen to hold by descending induction starting from the case $s=a$ (which is immediate). In the general case, supposing by induction that the case $s$ holds,  we have that
 $$
 \frac{(q;q)_{a}}{(q;q)_{s}}
 = 1 -   \sum_{i=s+1}^{a} q^{i}  \frac{ (q;q)_{a}}{ (q;q)_{i}}= 1 -   \sum_{i=s}^{a} q^{i}  \frac{ (q;q)_{a}}{ (q;q)_{i}} +    q^{s}  \frac{ (q;q)_{a}}{ (q;q)_{s}}.
 $$
 Hence, owing to $(1-q^s)(q;q)_{s-1}=(q;q)_s$, we have that
 $$
1 -   \sum_{i=s}^{a} q^{i}  \frac{ (q;q)_{a}}{ (q;q)_{i}} =     \frac{(q;q)_{a}}{(q;q)_{s}}- q^{s}  \frac{ (q;q)_{a}}{ (q;q)_{s}}=  \frac{(q;q)_{a}}{(q;q)_{s-1}},
 $$
 which proves \eqref{eqsumT2} by induction.

 Now suppose that $s=0$. In this case, the last $a_m+1$ letters in $T'$ are ordered from left to right and thus do not contribute to  ${\rm maj}_{\tilde \alpha'}$ and  ${\rm maj}_{\gamma^{(m)}}$. Hence, we have to show that
$$
  \sum_{c_m=0}^{a_m} q^{a_m-c_m}  \frac{ (q;q)_{a_m}}{ (q;q)_{a_m-c_m}}
  =  1 .
$$
But this holds given that it is easily seen to be the case $s=0$ of \eqref{eqsumT}.

Note that the $T'$ on the right-hand-side of \eqref{eqsumT} can be thought as stemming from
a pair $(P,Q)$ corresponding to the case $c_m=0$, in which case
$\tilde \alpha=(\lambda_1,\dots,\lambda_\ell,a_1-c_1,\dots,a_{m-1}-c_{m-1},a_m,1^{c_1+1},\dots,1^{c_{m-1}+1},1)$. Comparing with Definition~\ref{defTa}, it is thus natural to let $T'=T^{(m)}$ on the right-hand-side.
We have thus proven that
$$
\sum_{c_{m}=0}^{a_{m}}  q^{a_{m}-c_{m}} \frac{ (q;q)_{a_{m}}}{ (q;q)_{a_{m}-c_{m}}}
\sum_{P} \sum_{Q\in \mathcal  S_{\Omega \Gamma}} q^{{\rm maj}_{\tilde \alpha'}(T_{P,Q}')} 
= \sum_{P} \sum_{Q\in \mathcal  S_{\Omega \Gamma^{(m)}}} q^{{\rm maj}_{\gamma^{(m)}}(T^{(m)})},
$$
where $\Gamma^{(m)}=(c_1,\cdots,c_{m-1},0;{\rm sh}(P))$.

Following the same procedure, we can show that
$$
\sum_{c_{m-1}=0}^{a_{m-1}}  q^{a_{m-1}-c_{m-1}} \frac{ (q;q)_{a_{m-1}}}{ (q;q)_{a_{m-1}-c_{m-1}}}
 \sum_{P} \sum_{Q\in \mathcal  S_{\Omega \Gamma^{(m)}}}
 q^{{\rm maj}_{\gamma^{(m)}}(T^{(m)})} =
  \sum_{P} \sum_{Q\in \mathcal  S_{\Omega \Gamma^{(m-1)}}}
 q^{{\rm maj}_{\gamma^{(m-1)}}(T^{(m-1)})},
 $$
where $\Gamma^{(m-1)}=(c_1,\cdots,c_{m-12},0,0;{\rm sh}(P))$,
  and where
 $$
\gamma^{(m-1)}=(\lambda_1,\dots,\lambda_\ell,a_1-c_1,\dots,a_{m-2}-c_{m-2},1^{c_1+1},\dots,1^{c_{m-2}+1},a_{m-1}+1,a_m+1).
$$

After doing all the summations, we get
$$
\sum_{c_1=0}^{a_1} \dots \sum_{c_m=0}^{a_m} q^{|\pmb a|-|\pmb c|}
\frac{(q;q)_{a_1}\cdots (q;q)_{a_m}}{(q;q)_{a_1-c_1} \cdots (q;q)_{a_m-c_m}} 
 \sum_{P} \sum_{Q\in \mathcal  S_{\Omega \Gamma}} q^{{\rm maj}_{\tilde \alpha'}(T_{P,Q}')}
 q^{{\rm maj}_{\tilde \alpha}(T)}=\sum_{P} \sum_{Q\in \mathcal  S_{\Omega \Gamma^{(1)}}}
 q^{{\rm maj}_{\gamma^{(1)}}(T^{(1)})},
 $$
 where $\Gamma^{(1)}=(0^m;{\rm sh}(P))$, and where
 $$
\gamma^{(1)}=(\lambda_1,\dots,\lambda_\ell,a_1+1,\dots,a_m+1).
$$
Note that $\gamma^{(1)}= \tilde \Lambda$, and that $T^{(1)}=T_{\pmb a}$.
We thus finally have that
$$ 
\sum_{P} \sum_{Q\in \mathcal  S_{\Omega \Gamma^{(1)}}}
 q^{{\rm maj}_{\gamma^{(1)}}(T^{(1)})} = \sum_{T: \,{\rm sh}(T)=\Omega} q^{{\rm maj}_{\tilde \Lambda}(T_{\pmb a})} 
 $$
 since $Q$ is empty, which means that $(P,Q)=(P,\emptyset)$ now corresponds to an arbitrary standard filling $T$ of $\Omega$. This proves
 the theorem.
\end{proof}

\begin{remark}
  It was shown in \cite{L} that $K_{\Omega \Lambda}(q,t)=K_{\mu \lambda}(q,t)$ when
  $\Omega=(0^m;\mu)$ and $\Lambda=(0^m;\lambda)$. For such  $\Omega$ and $\Lambda$, it is easy to see that the combinatorial interpretation in Theorem~\ref{theo} corresponds to that of \eqref{usualcase}. Indeed, we have in this case that $\tilde \Lambda=(\lambda_1,\dots,\lambda_\ell,1^m)$ and $T_{\pmb a}=T_\circ$, where we recall that $T_\circ$ corresponds to $T$ with the $i$-circle transformed into a squared filled with the letter $|\lambda|+i$.  But since the last $m$ letters in $T_\circ$ do not contribute in this case to   $maj_{\tilde \Lambda}(T_\circ)$, we have that
  $maj_{\tilde \Lambda}(T_\circ)=maj_\lambda(P)$, where $P$ is the standard tableau corresponding to $T$ without its circles.
\end{remark}

\end{document}